\documentclass[10pt, a4paper, reqno, twoside]{amsart}

\usepackage{amsmath, amsfonts, amsthm, amssymb, mathtools, multicol, scalefnt, relsize, mathbbol, graphicx, enumitem, xcolor, slashed, todonotes, xfrac, fullpage, autobreak, lipsum}
\usepackage[mathscr]{euscript}
\usepackage[T1]{fontenc}
\usepackage[latin1]{inputenc}
\usepackage[all]{xy}
\setlist[itemize]{itemsep=0pt, parsep=0pt, partopsep=0pt, topsep=1pt, leftmargin=30pt}

\usepackage[hypertexnames=false,
backref=page,
    pdftex,
    pdfpagemode=UseNone,
    breaklinks=true,
    extension=pdf,
    colorlinks=true,
    linkcolor=blue,
    citecolor=blue,
    urlcolor=blue,
]{hyperref}

\newcommand\bcdot{\ensuremath{
  \mathchoice {\mskip\thinmuskip\lower0.2ex\hbox{\scalebox{1.6}{$\cdot$}}\mskip\thinmuskip}}{\mskip\thinmuskip\lower0.2ex\hbox{\scalebox{1.6}{$\cdot$}}\mskip\thinmuskip}
   {\lower0.3ex\hbox{\scalebox{1.2}{$\cdot$}}}
   {\lower0.3ex\hbox{\scalebox{1.2}{$\cdot$}}}
}

\theoremstyle{plain}
\newtheorem{theorem}{Theorem}[section]
\newtheorem{lemma}[theorem]{Lemma}
\newtheorem{corollary}[theorem]{Corollary}
\newtheorem{proposition}[theorem]{Proposition}

\theoremstyle{definition}
\newtheorem{definition}[theorem]{Definition}

\newtheorem{remark}[theorem]{Remark}
\newtheorem{example}[theorem]{Example}

\theoremstyle{plain}
\newtheorem{thmint}{Theorem}

\DeclareSymbolFontAlphabet{\mathbb}{AMSb}
\DeclareSymbolFontAlphabet{\mathbbl}{bbold}

\makeatletter
\@namedef{subjclassname@2020}{\textup{2020} Mathematics Subject Classification}
\makeatother

\allowdisplaybreaks[4]

\title[]{On Bismut--Ambrose--Singer manifolds}

\author{Giuseppe Barbaro}
\address{Giuseppe Barbaro\newline
		\indent Instituto de Ciencias Matem\'aticas (ICMAT), CSIC - UAM - UC3M - UCM\newline 
			\indent Nicol\'as Cabrera 13--15, 28049 Madrid, Spain}
\email{giuseppe.barbaro@icmat.es}

\author{Francesco Pediconi}
\address{Francesco Pediconi \newline
    \indent Dipartimento di Scienze Matematiche ``Giuseppe Luigi Lagrange'' Politecnico di Torino, \newline
    \indent Corso Duca degli Abruzzi 24, 10129 Torino, Italy}
\email{francesco.pediconi@polito.it}

\thanks{All authors are members of GNSAGA of INdAM. The first author is supported by the European Union's Horizon Europe research and innovation programme under the Marie Sk{\l}odowska-Curie Actions (Grant agreement No. 101273232 - SURF FLOW; HORIZON-MSCA-2025-PF), and was supported by the DFF Sapere Aude grant ``Conformal geometry: metrics and cohomology".}

\begin{document}

\begin{abstract}
We investigate Bismut--Ambrose--Singer (BAS) manifolds, namely Hermitian manifolds whose Bismut connection has parallel torsion and parallel curvature. We first establish a canonical reduction theorem for complete, simply-connected BAS manifolds. We then classify simply-connected BAS manifolds in the three fundamental homogeneous settings: the compact case, the non-compact semisimple case, and the nilpotent case. Building on this, we construct BAS manifolds in which these three geometries are combined, generalizing all previously known examples. Finally we classify complete, simply-connected, pluriclosed BAS manifolds.
\end{abstract}

\maketitle

\section{Introduction}
\setcounter{equation} 0

The aim of this paper is to investigate the geometry of a distinguished class of Hermitian manifolds, called \emph{Bismut--Ambrose--Singer manifolds} (BAS, for short).

\smallskip

The Bismut connection is a fundamental tool in non-K{\"a}hler geometry. Indeed, it provides a bridge between Riemannian and complex geometry, as it is the unique Hermitian connection with totally skew-symmetric torsion \cite{MR1006380, MR1456265}. Metric connections with totally skew-symmetric torsion had already appeared implicitly in the last century, prior to the introduction of the Bismut connection (see, \emph{e.g.}, \cite{MR107275,36,37}), and have since been extensively used in the study of manifolds with non-integrable $\mathsf{G}$-structures (see, \emph{e.g.}, \cite{MR1928632, MR1958088, MR2047649, MR2038309, MR2067465, MR2114426, MR2265474, MR2322400, ivanov2023riemannian}).

A metric connection on a Riemannian manifold is called \emph{Ambrose--Singer} if it has parallel torsion and parallel curvature. By the Ambrose--Singer theorem \cite{MR0102842}, a Riemannian manifold is locally homogeneous if and only if it admits an Ambrose--Singer connection. In general, such connections are far from being unique, and there is no canonical choice among them. Remarkably, requiring the Levi-Civita connection to be Ambrose--Singer defines the class of \emph{locally symmetric spaces}, which were classified by {\'E}.\ Cartan.

The Ambrose--Singer theorem admits a natural extension to the Hermitian setting, proved by Sekigawa \cite{MR0509406}, which states that a Hermitian manifold is locally Hermitian homogeneous if and only if it admits a Hermitian Ambrose--Singer connection. Among the possible Hermitian connections, the natural choice is given by the Bismut connection, leading to the class of BAS manifolds introduced in \cite{MR4630788, NiZh23}, namely, those Hermitian manifolds whose Bismut connection is Ambrose--Singer. The BAS condition can be regarded as the natural non-K\"ahler generalization of Hermitian locally symmetric spaces, and provides a rich and flexible class of examples. In contrast, the alternative canonical choice of connection, namely the Chern connection, leads to the class of \emph{Chern--Ambrose--Singer manifolds}, which, when complete and simply-connected, split as products of a complex Lie group and Hermitian symmetric spaces \cite[Theorem 1.2]{NiZh23}.

For the Bismut connection, the condition of parallel torsion is already relevant in itself. Indeed, manifolds with Bismut parallel torsion provide a rich class of special Hermitian manifolds that includes all {\em Bismut K\"ahler-like} \cite[Theorem 1]{MR4554474} and all {\em Vaisman} manifolds \cite[Corollary 3.8]{MR4480223}. More generally, metric connections with parallel totally skew-symmetric torsion play an important role in Riemannian geometry and have been widely studied (see, \emph{e.g.}, \cite{MR2047649, MR3319112, MR4184296}). Following \cite{MR4184296}, the existence of such a connection gives rise locally to a Riemannian submersion, whose fibres are naturally endowed with an Ambrose--Singer connection with totally skew-symmetric torsion. This provides further motivation for the study of BAS manifolds, as a natural step toward understanding Hermitian manifolds with Bismut parallel torsion.

\smallskip

The first main result of the paper provides a \emph{canonical reduction} for complete, simply-connected BAS manifolds. More concretely, we show that the Lie algebra of holomorphic Killing vector fields that are invariant under the identity component of the holomorphic isometry group gives rise to an abelian principal bundle. Accordingly, we say that a BAS manifold is \emph{reduced} if it admits no non-trivial invariant holomorphic Killing vector fields (see Definition \ref{def:reduced}).

\begin{thmint} \label{thm:MAIN-general}
Every complete, simply-connected Bismut--Ambrose--Singer manifold admits an abelian principal bundle structure over a complete, simply-connected, reduced Bismut--Ambrose--Singer manifold, such that the bundle projection is a holomorphic Riemannian submersion with totally geodesic fibres.
\end{thmint}

To prove Theorem \ref{thm:MAIN-general}, we first show that, on a complete, simply-connected BAS manifold $(M,J,g)$, the Bismut connection determines a canonical reductive complement at each point for the action of the identity component $\mathrm{Aut}^\mathrm{o}(M)$ of the full group of holomorphic isometries. With respect to this complement, $(M,g)$ is \emph{naturally reductive} (see Definition \ref{def:natred} and Theorem \ref{thm:BASnatrad}). We then consider the connected component $\mathsf{Z}(M)$ of the centre of $\mathrm{Aut}^\mathrm{o}(M)$, and show that its action on $M$ defines an abelian principal bundle
$$
\pi : M \to B = M/\mathsf{Z}(M) \,\, .
$$
Using the canonical reductive decomposition associated with the Bismut connection, we prove that $\pi$ is a holomorphic Riemannian submersion with totally geodesic fibres, and that the induced Hermitian structure on the base is again BAS and reduced (see Theorem \ref{thm:generalstructure}). Finally, since the Lee vector field $\theta^{\sharp}$ of a BAS manifold is invariant, holomorphic and Killing (see Proposition \ref{prop:thetaholkill}), every complete, simply-connected, reduced BAS manifold is \emph{balanced}. \smallskip

Among Riemannian homogeneous metrics, the naturally reductive ones form one of the most well-behaved classes. They have particularly nice geometric properties, while still being sufficiently rich to be of interest. For this reason, they have been extensively studied in the literature (see, \emph{e.g.}, \cite{MR0474145, MR0519928, MR0740188, MR0787113, MR1952476, MR1922121, MR3319112}). By \cite[Theorem 3.1]{MR0787113} and Theorem \ref{thm:BASnatrad}, if $(M,J,g)$ is a complete, simply-connected BAS manifold, and $\mathsf{N}$ denotes the nilradical of $\mathrm{Aut}^\mathrm{o}(M)$, then, for every semisimple Levi factor $\mathsf{G} = \mathsf{G}_{\mathrm{c}} \mathsf{G}_{\mathrm{nc}} \subset \mathrm{Aut}^{\mathrm{o}}(M)$, with $\mathsf{G}_{\mathrm{c}}$ compact, the subgroup $\mathsf{G}_{\mathrm{c}}  \mathsf{G}_{\mathrm{nc}} \mathsf{N}$ acts transitively on $M$. Thus, the study of BAS manifolds is partially reduced to the following three cases: those admitting a transitive action by a compact semisimple group $\mathsf{G}_{\mathrm{c}}$, by a semisimple group of non-compact type $\mathsf{G}_{\mathrm{nc}}$, or by a nilpotent group $\mathsf{N}$. Following this strategy, we classify complete, simply-connected BAS manifolds in each of these three cases, combining the results of \cite{MR0787113} with the additional rigidity imposed by the complex structure. 
\smallskip

First, the classification of compact, simply-connected BAS manifolds is as follows.

\begin{thmint} \label{thm:MAIN-cpt}
Let $(M,J,g)$ be a compact, simply-connected Hermitian manifold. Then $(M,J,g)$ is Bismut--Ambrose--Singer if and only if $(M,J)$ is the total space of a homogeneous holomorphic principal torus bundle over a flag manifold, and $g$ is homogeneous, invariant under the right torus action, and induces a normal homogeneous metric on the base.
\end{thmint}

The principal torus bundle appearing in the statement above coincides with the canonical reduction given by Theorem \ref{thm:MAIN-general}. Notice that we allow the torus to be trivial, so that the reduced manifolds in this classification are precisely the flag manifolds endowed with normal homogeneous metrics. Among these, the K{\"a}hler manifolds are precisely the compact Hermitian symmetric spaces.

\smallskip

Secondly, we classify simply-connected BAS manifolds admitting a transitive action by a semisimple Lie group of non-compact type.

\begin{thmint} \label{thm:MAIN-noncpt}
Let $(M,J,g)$ be a simply-connected Hermitian manifold which does not split as a non-trivial Hermitian product, and assume that there exists a semisimple Lie group of non-compact type acting transitively on $M$ by holomorphic isometries. Then $(M,J,g)$ is Bismut--Ambrose--Singer if and only if one of the following holds:
\begin{itemize}
\item[$i)$] $(M,J)$ is a complex simple Lie group and $g$ is a multiple of its canonical metric;
\item[$ii)$] $(M,J)$ is the total space of a homogeneous holomorphic principal abelian bundle over a Hermitian symmetric space of non-compact type, and $g$ is homogeneous, invariant under the right abelian action, and induces the symmetric metric on the base.
\end{itemize}
\end{thmint}

The manifolds in case $i)$ are reduced and non-K{\"a}hler. In case $ii)$, the principal abelian bundle coincides with the canonical reduction given by Theorem \ref{thm:MAIN-general}, and its fibres are contractible. In particular, these bundles are smoothly trivial, although the Hermitian structure does not split as a product. Notice that we allow the abelian group to be trivial, so that case $ii)$ also includes irreducible Hermitian symmetric spaces of non-compact type.

\smallskip

Finally, we classify simply-connected, nilpotent Lie groups that admit a left-invariant BAS structure.

\begin{thmint} \label{thm:MAIN-nil}
Let $(\mathsf{N},J,g)$ be a simply-connected nilpotent Lie group, with Lie algebra $\mathfrak{n}$, endowed with a left-invariant Hermitian structure. Assume that $(\mathsf{N},J,g)$ does not split as a non-trivial Hermitian product, and that $\mathfrak{n}$ is not abelian. Then $(\mathsf{N},J,g)$ is Bismut--Ambrose--Singer if and only if $\mathfrak{n}$ admits a $(J,g)$-unitary basis $\{z_1,Jz_1,\dots,z_{k},Jz_{k},e_1,Je_1,\dots,e_{m},Je_{m}\}$ such that $\{z_1,Jz_1,\dots,z_{k},Jz_{k}\}$ spans the centre $\mathfrak{z}(\mathfrak{n})$ of $\mathfrak{n}$ and the only non-zero Lie brackets are
$$
[e_{j},Je_{j}] = \sum_{i=1}^{k}\alpha_j^i z_i+\beta_j^i Jz_i, \qquad j=1,\dots,m,
$$
for some coefficients $\alpha_j^i, \beta_j^i \in \mathbb{R}$. Moreover, the canonical reduction of $(\mathsf{N},J,g)$ is given by the projection
$$
\pi:\mathsf{N}\to \mathsf{N}/\mathsf{Z}(\mathsf{N})
$$
with respect to the centre $\mathsf{Z}(\mathsf{N})$ of $\mathsf{N}$, and the induced Hermitian structure on the base is K\"ahler flat.
\end{thmint}

Using the classification in Theorem \ref{thm:MAIN-nil}, we provide the explicit list of nilpotent Lie algebras admitting a BAS structure in low dimensions (see Corollary \ref{cor:BASnil}). We emphasize that, by virtue of \cite[Proposition 1.8]{MR4920800}, Theorem \ref{thm:MAIN-nil} classifies all simply-connected nilpotent Lie groups admitting a left-invariant Hermitian structure with nilpotent complex structure, in the sense of \cite{MR1665327}, and parallel Bismut torsion.
\smallskip

In the general case, none of the three subgroups $\mathsf{G}_{\mathrm{c}}, \mathsf{G}_{\mathrm{nc}}, \mathsf{N} \subset \mathrm{Aut}^{\mathrm{o}}(M)$ needs to act transitively on $M$. Their orbits are nevertheless naturally reductive and totally geodesic with respect to the induced metric (see \cite[Proposition 2.8]{MR0787113}). Since the complex structure need not preserve the tangent spaces of these orbits, they do not necessarily inherit a BAS structure from $M$. This phenomenon occurs in the examples constructed in Proposition \ref{prop:BAS-NGH}, where BAS manifolds are obtained by combining a nilpotent factor with compact and non-compact homogeneous factors. This construction generalizes all previously known examples in the literature \cite{MR4920800, NiZh23, ZZbal, BM26}. Another instance in which the three geometries appear simultaneously is provided by Proposition \ref{prop:BAS-SKT}, where we classify complete, simply-connected Hermitian manifolds that are both BAS and \emph{pluriclosed} by exploiting the results of \cite{MR5124088}.
\smallskip

The paper is organized as follows. Section \ref{sect:prel} collects background material on naturally reductive manifolds and on the Bismut connection. In Section \ref{sect:AS}, we prove that the BAS condition is equivalent to the Hermitian naturally reductive condition (see Theorem \ref{thm:BASnatrad}). In Section \ref{sect:reductive}, we prove Theorem \ref{thm:MAIN-general}. In Section \ref{sect:compact}, we prove Theorem \ref{thm:MAIN-cpt}. In Section \ref{sect:noncompact}, we prove Theorem \ref{thm:MAIN-noncpt}. In Section \ref{sect:nil}, we prove Theorem \ref{thm:MAIN-nil}. Finally, in Section \ref{sect:final} we focus on the construction of BAS manifolds, and we classify complete, simply-connected, pluriclosed BAS manifolds.
\medskip

\noindent {\itshape Acknowledgements.\ } The authors would like to thank Beatrice Brienza, Jorge Lauret and Lei Ni for their comments and interest.

\medskip
\section{Preliminaries}
\label{sect:prel} \setcounter{equation} 0

\subsection{The Bismut connection} \hfill \par

Let $(M,J,g)$ be a complete Hermitian manifold of complex dimension $n$ and denote by $\omega \coloneqq g(J\cdot,\cdot)$ the associated fundamental $2$-form. Given a vector bundle $E \to M$ over $M$, we denote by $\Gamma(E)$ the space of its smooth sections. We further denote by $\mathrm{Aut}^\mathrm{o}(M)$ the group of holomorphic isometries of $(M,J,g)$ isotopic to the identity, and by
$$
\mathfrak{aut}(M) \coloneqq \big\{V \in \Gamma(TM) : \mathcal{L}_Vg = 0 \, , \,\, \mathcal{L}_VJ = 0 \big\}
$$
its Lie algebra. Let $D$ be the Levi-Civita connection of $(M,g)$ and let $\nabla$ be the \emph{Bismut connection} of $(M,J,g)$, defined by
\begin{equation} \label{eq:Bismutconn}
\nabla_XY \coloneqq D_XY +\tfrac12T(X,Y) \,\, , \quad g(T(X,Y),Z) \coloneqq {\rm d}\omega(JX,JY,JZ) \,\, .
\end{equation}
The tensor $T$ coincides with the torsion of $\nabla$, namely,
\begin{equation} \label{eq:T}
T(X,Y) = \nabla_XY -\nabla_YX -[X,Y] \,\, .
\end{equation}
We recall that, by \cite{MR1006380, MR1456265}, the Bismut connection is characterised as the unique linear connection with \emph{totally skew-symmetric torsion} which preserves the Hermitian structure, \emph{i.e.}, $\nabla J = \nabla g=0$. We also denote by $\mathrm{Rm}$ the Riemannian curvature tensor and by $R$ the \emph{Bismut curvature tensor}, \emph{i.e.},
$$\begin{aligned}
{\rm Rm}(X,Y).Z &\coloneqq D_{[X,Y]}Z -D_XD_YZ +D_YD_XZ \,\, , \\
R(X,Y).Z &\coloneqq \nabla_{[X,Y]}Z -\nabla_X\nabla_YZ +\nabla_Y\nabla_XZ \,\, .
\end{aligned}$$
We recall the following identities.

\begin{lemma}
Let $(M,J,g)$ be a Hermitian manifold. Then
\begin{equation} \label{eq:TDJ}
g(T(X,Y),Z) = -g((D_{JX}J)Y,Z) -g((D_{JY}J)Z,X) -g((D_{JZ}J)X,Y) \,\, .
\end{equation}
Moreover, if $\nabla T = 0$, then
\begin{equation} \label{eq:R-Rm}
R(X,Y).Z -{\rm Rm}(X,Y).Z = -\tfrac12T(T(X,Y),Z) -\tfrac14T(T(Y,Z),X) -\tfrac14T(T(Z,X),Y) \,\, .
\end{equation}
\end{lemma}

\begin{proof}
We first notice that \eqref{eq:TDJ} follows by a straightforward computation based on \eqref{eq:Bismutconn}. Moreover, \eqref{eq:R-Rm} follows, \emph{e.g.}, by \cite{MR1822270}.
\end{proof}

\smallskip

For later use, we recall that the \emph{Chern connection} of $(M,J,g)$ is defined by
\begin{equation*} \label{eq:Chernconn}
g(\nabla^{\mathrm{Ch}}_XY,Z) \coloneqq g(D_XY,Z) -\tfrac12\mathrm{d}\omega(JX,Y,Z)
\end{equation*}
and it is the unique Hermitian connection whose torsion 
\begin{equation} \label{eq:defTCh}
-2g(T^{\rm Ch}(X,Y),Z) = \mathrm{d}\omega(JX,Y,Z) +\mathrm{d}\omega(X,JY,Z)
\end{equation}
has vanishing $(1,1)$-component, \emph{i.e.},
\begin{equation} \label{eq:TCh11}
JT^{\rm Ch}(X,Y) = T^{\rm Ch}(JX,Y) = T^{\rm Ch}(X,JY) \,\, .
\end{equation}
We denote by $\theta \in \Gamma(T^*M)$ the \emph{Lee form} of $(M,J,g)$, defined locally as
\begin{equation} \label{def:theta}
\theta \coloneqq \sum_{\alpha=1}^{2n} g(T^{\mathrm{Ch}}(\,\cdot\,,e_{\alpha}),e_{\alpha}) \,\, ,
\end{equation}
where $\{e_{\alpha}\}$ is a local unitary frame for $TM$. We also denote by $\theta^{\sharp} \in \Gamma(TM)$ its dual vector field, called \emph{Lee vector field}.

\subsection{Naturally reductive manifolds} \hfill \par

We recall that a (connected) Riemannian manifold $(M,g)$ is said to be \emph{locally homogeneous} if its pseudogroup of local isometries acts transitively on it, \emph{i.e.}, if for any choice of $x,y \in M$ there exist two open sets $\mathscr{U}, \mathscr{V} \subset M$ and a local isometry $f: \mathscr{U} \to \mathscr{V}$ such that $x \in \mathscr{U}$, $y \in \mathscr{V}$ and $f(x) = y$. We recall the following characterisation for local homogeneity (see also \cite[Theorem 2.1]{MR1261452}).

\begin{theorem}[\cite{MR0102842}] \label{thm:AS}
A Riemannian manifold $(M,g)$ is locally homogeneous if and only if it admits a metric connection with parallel torsion and curvature. Any such connection is called an \emph{Ambrose--Singer connection}.
\end{theorem}

By \cite[Main Theorem]{MR0131248}, if $(M,g)$ is locally homogeneous and complete, then its universal Riemannian cover $(\widetilde{M},g)$ is homogeneous, and hence it is equivariantly diffeomorphic to a (not necessarily unique) quotient $\mathsf{G}/\mathsf{H}$, where $\mathsf{G}$ is a connected, closed group of isometries acting transitively on $\widetilde{M}$ and $\mathsf{H}$ is the compact subgroup of $\mathsf{G}$ that fixes a distinguished point $p \in \widetilde{M}$. The Riemannian homogeneous space $(\widetilde{M} = \mathsf{G}/\mathsf{H}, g)$ is necessarily \emph{reductive}, \emph{i.e.}, the Lie algebra $\mathfrak{h}$ of $\mathsf{H}$ admits an ${\rm Ad}(\mathsf{H})$-invariant complement $\mathfrak{m}$ inside the Lie algebra $\mathfrak{g}$ of $\mathsf{G}$ (see \cite[paragraph 7.22]{MR2371700}). We call $\mathfrak{m}$ a \emph{reductive complement at $p$ for $\mathsf{G}$}. The evaluation map
\begin{equation} \label{eq:evmap}
\mathfrak{m} \to T_{p}\widetilde{M} \,\, , \quad X \mapsto X^*_{p} \coloneqq \tfrac{\rm d}{{\rm d}t} \exp(tX) \cdot p \,\big|_{t=0}
\end{equation}
induces a canonical identification $\mathfrak{m} \simeq T_{p}\widetilde{M}$ and, more generally, a canonical bijection between the space of $\mathsf{G}$-invariant tensor fields on $\widetilde{M}$ and the space of $\mathrm{Ad}(\mathsf{H})$-invariant tensors on $\mathfrak{m}$. For notational convenience, we use the same symbol for a $\mathsf{G}$-invariant tensor field and for its corresponding $\mathrm{Ad}(\mathsf{H})$-invariant tensor on $\mathfrak{m}$. We also recall that
\begin{equation} \label{eq:antiiso}
[X^*,Y^*]|_p = -[X,Y]_{\mathfrak{m}}^*|_p \quad \text{for every $X, Y \in \mathfrak{m}$} \,\, ,
\end{equation}
where the subscript denotes the projection onto $\mathfrak{m}$ with respect to the direct sum decomposition of vector spaces $\mathfrak{g} = \mathfrak{h} +\mathfrak{m}$.

\begin{definition} \label{def:natred}
Let $(M,g)$ be a complete, locally homogeneous Riemannian space and denote by $(\widetilde{M},g)$ its universal Riemannian cover. Let also $\mathsf{G}$ be a transitive group of isometries of $(\widetilde{M},g)$, $p \in \widetilde{M}$ a point and $\mathfrak{m}$ a reductive complement at $p$ for $\mathsf{G}$. Then, $(M,g)$ is called \emph{$(\mathsf{G}, p, \mathfrak{m})$-naturally reductive} if
\begin{equation} \label{eq:natrad} \tag{NR}
g([X,Y]_{\mathfrak{m}},Z) +g(Y,[X,Z]_{\mathfrak{m}}) = 0 \quad \text{for every $X, Y, Z \in \mathfrak{m}$} \,\, .
\end{equation}
We say that $(M,g)$ is \emph{$\mathsf{G}$-naturally reductive} if it is $(\mathsf{G}, p, \mathfrak{m})$-naturally reductive for some choice of $p$ and $\mathfrak{m}$ as above.
\end{definition}

As a corollary of Theorem \ref{thm:AS}, it follows that a complete Riemannian manifold is naturally reductive if and only if it admits an Ambrose--Singer connection with totally skew-symmetric torsion. Indeed, the existence of such a connection follows from the naturally reductive condition \eqref{eq:natrad} by \cite[Ch.\ X, Theorem 2.1]{MR1393941}. Conversely, the naturally reductive condition \eqref{eq:natrad} follows from the existence of an Ambrose--Singer connection with totally skew-symmetric torsion via the so-called \emph{Nomizu construction} \cite{MR0059050} (see also \cite{MR3921240}), which we shall discuss in the next section for the Hermitian case.

\medskip
\section{Hermitian naturally reductive manifolds} 
\label{sect:AS} \setcounter{equation} 0

We fix some notation. Let us denote by $\mathcal{T}(\mathbb{R}^{2n})$ the space of tensors on $\mathbb{R}^{2n}$. If $v \in \mathbb{R}^{2n}$ is a vector and $\tau \in \mathcal{T}(\mathbb{R}^{2n})$ is of type $(r,s)$ with $s \geq 1$, we define $v\,\lrcorner\,\tau$ to be the $(r,s-1)$-tensor given by
$$
(v\,\lrcorner\,\tau)(w_1,{\dots},w_{s-1}) \coloneqq \tau(v,w_1,{\dots},w_{s-1}) \,\, .
$$
Moreover, we denote by
$$
\mathfrak{gl}(2n,\mathbb{R}) \times \mathcal{T}(\mathbb{R}^{2n}) \to \mathcal{T}(\mathbb{R}^{2n}) \,\, , \quad (A, \tau) \mapsto A \cdot \tau
$$
the Lie algebra action defined as the differential of the natural action of $\mathsf{GL}(2n,\mathbb{R})$ on $\mathcal{T}(\mathbb{R}^{2n})$ induced by change of basis. A direct computation shows that
\begin{equation} \label{eq:commact}
[A,B]\cdot \tau = A \cdot B \cdot \tau - B \cdot A \cdot \tau  \quad \text{for every $A,B \in \mathfrak{gl}(2n,\mathbb{R})$ and $\tau \in \mathcal{T}(\mathbb{R}^{2n})$}\,\, ,
\end{equation}
where $[A,B] = A \circ B - B \circ A$ denotes the commutator in $\mathfrak{gl}(2n,\mathbb{R})$. These linear-algebraic notions can be naturally formulated on smooth manifolds.

\subsection{The Bismut--Ambrose--Singer condition} \hfill \par

Let $(M,J,g)$ be a Hermitian manifold. Both the Levi-Civita connection $D$ and the Bismut connection $\nabla$ act on tensor fields on $M$ in a natural way, and they are related by the following identity.

\begin{lemma}
Let $(M,J,g)$ be a Hermitian manifold. If $\tau$ is an $(r,s)$-tensor field and $V$ is a vector field on $M$, then
\begin{equation} \label{eq:derivations}
\nabla_V \tau = D_V \tau + \frac{1}{2} (V \lrcorner T) \cdot \tau \,\, .
\end{equation}
\end{lemma}
\begin{proof}
Notice that both $\nabla_V$ and $D_V + \big(\tfrac12 V \lrcorner T\big) \cdot$ are derivations and that they coincide on functions and vector fields. Therefore, \eqref{eq:derivations} follows from \cite[Lemma, p.\ 30]{MR1393940}.
\end{proof}

We prove now the following technical lemma on Hermitian manifolds with parallel Bismut torsion.

\begin{lemma} \label{lem:nablaDk}
Let $(M,J,g)$ be a Hermitian manifold with parallel Bismut torsion, namely $\nabla T = 0$, and let $\tau $ be a $\nabla$-parallel $(r,s)$-tensor field on $M$. Then, the following identity holds
\begin{equation} \label{eq:Dktau}
D^k_{V_1,{\dots},V_k}\tau = -\frac{1}{2}(V_1 \lrcorner T) \cdot D^{k-1}_{V_2,{\dots},V_k}\tau
+\frac{1}{2}\sum_{i=2}^k D^{k-1}_{V_2,{\dots},T(V_1,V_i),{\dots},V_k}\tau
\end{equation}
for all $V_1,{\dots},V_k \in \Gamma(TM)$, for all $k \geq 1$. In particular, $\nabla(D^k\tau) = 0$ for all $k \geq 0$.
\end{lemma}

\begin{proof}
We recall the following well-known fact: for any point $p \in M$ and any vector $v \in T_pM$, there exists a vector field $V$ defined on a neighborhood of $p$ such that $V|_p=v$ and $\nabla V|_p = 0$. Therefore, in proving \eqref{eq:Dktau} by induction, we may restrict to a point $p$ and assume that $\nabla V_i |_p = 0$.

When $k =1$, \eqref{eq:Dktau} follows from \eqref{eq:derivations}. In particular, this shows that $\nabla(D\tau) = 0$. 
Fix then $k \geq 1$ and assume that \eqref{eq:Dktau} holds and that $\nabla(D^k\tau) = 0$. By \eqref{eq:derivations}, we get
$$\begin{aligned}
0 &= \nabla_{V_1}(D^k_{V_2,{\dots},V_{k+1}}\tau) \\
&= D_{V_1}(D^k_{V_2,{\dots},V_{k+1}}\tau) +\frac{1}{2}(V_1 \lrcorner T) \cdot D^k_{V_2,{\dots},V_{k+1}}\tau \\
&= D^{k+1}_{V_1,{\dots},V_{k+1}}\tau -\frac{1}{2}\sum_{i=2}^{k+1} D^k_{V_2,{\dots},T(V_1,V_i),{\dots},V_{k+1}}\tau +\frac{1}{2}(V_1 \lrcorner T) \cdot D^k_{V_2,{\dots},V_{k+1}}\tau 
\end{aligned}$$
and, again, $\nabla(D^{k+1}\tau) = 0$.
\end{proof}

We now introduce the main object of study of this paper.

\begin{definition}
A Hermitian manifold $(M,J,g)$ is called \emph{Bismut--Ambrose--Singer} (\emph{BAS} for short) if its Bismut connection $\nabla$ is Ambrose--Singer, namely, $\nabla R = \nabla T = 0$.
\end{definition}

By using Lemma \ref{lem:nablaDk}, the BAS condition can be characterised in terms of the tensor fields $J$ and ${\rm Rm}$ as follows.

\begin{proposition} \label{prop:BASJRm}
A Hermitian manifold $(M,J,g)$ is BAS if and only if $\nabla (DJ) = \nabla{\rm Rm} = 0$. In particular, in this case, it holds
\begin{equation} \label{eq:nJnRm}
\nabla(D^kJ) = \nabla(D^k{\rm Rm}) = 0 \quad \text{ for all $k \geq 0$} \,\, .
\end{equation}
\end{proposition}

\begin{proof}
Assume that $(M,J,g)$ is BAS. Then, $\nabla {\rm Rm} = 0$ by \eqref{eq:R-Rm}, and $\nabla(DJ)=0$ by Lemma \ref{lem:nablaDk}. On the other hand, assume that $\nabla(DJ) = \nabla{\rm Rm} = 0$. Then, $\nabla T =0$ by \eqref{eq:TDJ}, and so $\nabla R =0$ by \eqref{eq:R-Rm}. Finally, if $(M,J,g)$ is BAS, then \eqref{eq:nJnRm} follows by Lemma \ref{lem:nablaDk}.
\end{proof}

\subsection{The Nomizu construction for Bismut--Ambrose--Singer manifolds} \label{sect:Nomcons} \hfill \par

Let $(M,J,g)$ be a complete, simply-connected BAS manifold.
Following \cite{MR0059050} (see also \cite{MR3921240}), the Nomizu construction associates a distinguished Lie algebra $\mathfrak{N}$ to the Bismut connection $\nabla$, so that $(M,g)$ is naturally reductive with respect to the corresponding simply-connected Lie group. Here, we adapt this construction in a natural way to the Hermitian context. Specifically, we require that the isotropy subalgebra $\mathfrak{N}_0$ of $\mathfrak{N}$ preserves $J$. Consequently, we prove that the Lie algebra $\mathfrak{N}$ is isomorphic to $\mathfrak{aut}(M)$, and so $M$ is $\mathrm{Aut}^\mathrm{o}(M)$-naturally reductive. As a by-product, we obtain a canonical reductive complement, induced by the Bismut connection, for the action of $\mathrm{Aut}^\mathrm{o}(M)$ on $M$.

We begin with the following characterisation for the Lie algebra $\mathfrak{aut}(M)$, which is a straightforward consequence of \cite[Proposition 3.6]{MR4579179}. Recall that a complex manifold $(M,J)$ admits a real-analytic structure with respect to which $J$ is real-analytic, and that all the real-analytic structures on $M$ are equivalent up to a real-analytic diffeomorphism.

\begin{proposition} \label{prop:Nomcons1}
Let $(M,J,g)$ be a complete, simply-connected Hermitian manifold, and let $p \in M$ be a point. The set
\begin{multline*}
\mathfrak{H}(p) \coloneqq \big\{(\tilde{A},v) : v \in T_pM \, , \,\, \tilde{A} \in \mathfrak{so}(T_pM,g_p) \, , \\ \text{$v \lrcorner (D^{k+1}J)_p + \tilde{A} \cdot (D^kJ)_p = 0$\, and \,$v \lrcorner (D^{k+1}{\rm Rm})_p + \tilde{A} \cdot (D^k{\rm Rm})_p = 0$\, for all $k \geq 0$} \big\}
\end{multline*}
endowed with the brackets
$$
\big[(\tilde{A},v), (\tilde{B},w)\big] \coloneqq \big([\tilde{A},\tilde{B}] +{\rm Rm}_p(v,w), \tilde{A}.w -\tilde{B}.v\big)
$$
is a Lie algebra. Moreover, if $g$ is real-analytic, then the map
\begin{equation} \label{eq:psi}
\psi: \mathfrak{aut}(M) \to \mathfrak{H}(p) \,\, , \quad \psi(V) \coloneqq \big({-}DV|_p,V_p\big)
\end{equation}
is a Lie algebra isomorphism.
\end{proposition}

Every element of the Lie algebra $\mathfrak{H}(p)$ is called a \emph{holomorphic Killing generator at $p$}. This definition was introduced in \cite{MR4579179} as a refinement of the notion of Killing generators introduced in \cite{MR0119172}. Notice that, as a consequence of Proposition \ref{prop:Nomcons1}, the isotropy subalgebra of real holomorphic Killing vector fields that vanish at $p$, \emph{i.e.},
$$
\mathfrak{aut}_{0,p}(M) \coloneqq \big\{ V \in \mathfrak{aut}(M) : V_p = 0 \big\} \,\, ,
$$
is isomorphic via the map $\psi$ defined in \eqref{eq:psi} to the subalgebra
$$
\mathfrak{H}_0(p) \coloneqq \big\{(\tilde{A},v) \in \mathfrak{H}(p) : v = 0 \big\} \subset \mathfrak{H}(p) \,\, .
$$

\begin{proposition} \label{prop:Nomcons2}
Let $(M,J,g)$ be a complete, simply-connected BAS manifold and let $p \in M$ be a point. The set
\begin{equation*} \label{eq:defN0}
\mathfrak{N}_0(p) \coloneqq \big\{A \in \mathfrak{gl}(T_pM) : A \cdot g_p = A \cdot J_p = A \cdot T_p = A \cdot R_p = 0\big\}
\end{equation*}
is a Lie subalgebra of $\mathfrak{gl}(T_pM)$. Moreover, the formula
\begin{equation} \label{eq:Nomalg}
\mathfrak{N}(p) \coloneqq \mathfrak{N}_0(p) \oplus T_pM \,\, , \quad [(A,v), (B,w)] \coloneqq \big([A,B] + R_p(v,w), A.w -B.v +T_p(v,w)\big)
\end{equation}
defines a Lie algebra. Furthermore, the map
\begin{equation} \label{eq:phi}
\phi: \mathfrak{N}(p) \to \mathfrak{H}(p) \,\, , \quad \phi(A,v) \coloneqq \big(A +\tfrac12v\lrcorner T_p, v\big)
\end{equation}
is a Lie algebra isomorphism that sends $\mathfrak{N}_0(p)$ to $\mathfrak{H}_0(p)$.
\end{proposition}

\begin{proof}
The first claim is a direct consequence of \eqref{eq:commact}. We prove below that the map $\phi$ defined in \eqref{eq:phi} is an isomorphism of vector spaces. Then, one can verify via a straightforward computation that it is, in fact, a Lie algebra isomorphism with respect to the bracket defined in \eqref{eq:Nomalg}.

First, we prove that the image of $\phi$ lies in $\mathfrak{H}(p)$. Let $A \in \mathfrak{N}_0(p)$ and observe that $A \cdot J = 0$ by hypothesis and $A \cdot \mathrm{Rm} = 0$ by \eqref{eq:R-Rm}. Notice that, since $A \cdot T = 0$, it holds
\begin{equation} \label{eq:AT=0}
[A,v \lrcorner T] = (Av) \lrcorner T \quad \text{for all $v \in T_pM$}\,\, .
\end{equation}
Therefore, for every $\nabla$-parallel tensor field $\tau$, it follows that
$$\begin{aligned}
(A \cdot D^k\tau)_{V_1,{\dots},V_k} &\,\,= A \cdot (D^k\tau)_{V_1,{\dots},V_k} - \sum_{i=1}^k (D^k\tau)_{V_1,{\dots},AV_i{\dots},V_k} \\
&\overset{\eqref{eq:Dktau}}{=} -\frac{1}{2}A \cdot \big((V_1 \lrcorner T) \cdot D^{k-1}\tau\big)_{V_2,{\dots},V_k}
+\frac{1}{2} \big((AV_1 \lrcorner T) \cdot D^{k-1}\tau\big)_{V_2,{\dots},V_k} \\
&\quad\quad\quad +\frac{1}{2}\sum_{i=2}^k \big((V_1 \lrcorner T) \cdot D^{k-1}\tau\big)_{V_2,{\dots},AV_i,{\dots},V_k} \\
&\,\, = -\frac{1}{2} \big(A \cdot \big((V_1 \lrcorner T) \cdot D^{k-1}\tau\big)\big)_{V_2,{\dots},V_k}
+\frac{1}{2} \big((AV_1 \lrcorner T) \cdot D^{k-1}\tau\big)_{V_2,{\dots},V_k} \\
&\overset{\eqref{eq:AT=0}}{=}  -\frac{1}{2}\big((V_1 \lrcorner T) \cdot \big(A \cdot D^{k-1}\tau\big)\big)_{V_2,{\dots},V_k}
\end{aligned}$$
and so, if $A \cdot \tau = 0$, it follows by induction that $A \cdot D^k\tau = 0$ for all $k \geq 0$. This proves that $(A,0) \in \mathfrak{H}(p)$, and so $(A,0) \in \mathfrak{H}_0(p)$. Take now $v \in T_pM$ and notice that
$$
v \lrcorner (D^{k+1}\tau)_p + \big(\tfrac12v\lrcorner T_p\big)\cdot (D^k\tau)_p \overset{\eqref{eq:derivations}}{=} \nabla_v(D^k\tau) \,\, .
$$
Therefore, by \eqref{eq:nJnRm}, it follows that $\big(\tfrac12v\lrcorner T_p, v\big) \in \mathfrak{H}(p)$.

The map $\phi$ is clearly injective. We thus prove that it is also surjective. Fix $(\tilde{A},v) \in \mathfrak{H}(p)$ and define $A \coloneqq \tilde{A} -\tfrac12v\lrcorner T_p$. Then, a straightforward computation at the point $p$ shows that
$$\begin{aligned}
A \cdot (D^k\tau)_p
&= \tilde{A} \cdot (D^k\tau)_p
-\tfrac12(v\lrcorner T_p)\cdot(D^k\tau)_p\\
&= -v\lrcorner(D^{k+1}\tau)_p
-\tfrac12(v\lrcorner T_p)\cdot(D^k\tau)_p\\
&= -\nabla_v(D^k\tau)_p \,\, ,
\end{aligned}$$
for $\tau \in \{J, \mathrm{Rm}\}$. This shows that $A \cdot (D^kJ)_p = A \cdot (D^k\mathrm{Rm})_p = 0$ for all $k \geq 0$. Moreover, $A\cdot g_p=0$, since both $\tilde{A}$ and $v\lrcorner T_p$ are skew-symmetric. Hence, equations \eqref{eq:TDJ} and \eqref{eq:R-Rm} imply that
$A \cdot T_p = A \cdot R_p = 0$. Therefore, $A \in \mathfrak{N}_0(p)$, and so $\phi$ is an isomorphism.
\end{proof}

Using the previous proposition, we deduce the main result of this section.

\begin{theorem} \label{thm:BASnatrad}
Let $(M,J,g)$ be a complete, simply-connected Hermitian manifold, denote by $\mathrm{Aut}^\mathrm{o}(M)$ the identity component of its group of holomorphic isometries and let $p \in M$ be a point. If $(M,J,g)$ is BAS, then $(M,g)$ is $\big(\mathrm{Aut}^{\mathrm{o}}(M), p, \mathfrak{m}_p\big)$-naturally reductive, where
\begin{equation} \label{eq:redcompl}
\mathfrak{m}_p \coloneqq \big\{V \in \mathfrak{aut}(M) : -\nabla V|_p = V_p \lrcorner T_p \big\} \,\, .
\end{equation}
Conversely, if $(M,g)$ is $\mathsf{G}$-naturally reductive for some connected, transitive subgroup $\mathsf{G} \subset \mathrm{Aut}^{\mathrm{o}}(M)$, then $(M,J,g)$ is BAS.
\end{theorem}

\begin{proof}
Assume that $(M,J,g)$ is BAS. Since the Bismut connection is complete, the Ambrose--Singer Theorem \cite{MR0102842} (see also \cite{MR0509406}) implies that the Lie algebra $\mathfrak{N}(p)$ integrates to a transitive Lie group of holomorphic isometries of $(M,J,g)$. Hence, the metric $g$ is real-analytic (see, \emph{e.g.}, \cite[Lemma 3.5]{MR4579179}). Then, one can check that the set $\mathfrak{m}_p$ defined in \eqref{eq:redcompl} is a vector subspace of $\mathfrak{aut}(M)$ and that
$$
\mathfrak{m}_p = (\psi^{-1} \circ \phi)(T_pM) \,\, ,
$$
where $\phi$ has been defined in \eqref{eq:phi} and $\psi$ has been defined in \eqref{eq:psi}. As a consequence, $\mathfrak{m}_p$ is a reductive complement at $p$ for $\mathrm{Aut}^\mathrm{o}(M)$, and $\psi^{-1} \circ \phi|_{T_pM}$ coincides with the inverse of the evaluation map at $p$ defined in \eqref{eq:evmap}. Therefore, it follows that the action of $\mathrm{Aut}^\mathrm{o}(M)$ is transitive. Moreover, \eqref{eq:T} and \eqref{eq:redcompl} imply that
$$
T(V,W)\big|_p = [V,W]\big|_p \quad \text{ for every $V, W \in \mathfrak{m}_p$} \,\, .
$$
Therefore, \eqref{eq:natrad} follows from \eqref{eq:antiiso} and the fact that $T$ is totally skew-symmetric.

Conversely, let $\mathsf{G} \subset \mathrm{Aut}^{\mathrm{o}}(M)$ be a connected subgroup acting transitively on $M$, and assume that $(M,g)$ is $\mathsf{G}$-naturally reductive. Then, there exists a reductive complement $\mathfrak{m}$ at $p$ for $\mathsf{G}$. By \cite[Ch.\ X, Theorem 2.1 and Theorem 2.6]{MR1393941}, the reductive complement $\mathfrak{m}$ induces an Ambrose--Singer connection $\nabla'$, which has totally skew-symmetric torsion by \cite[Ch.\ X, Theorem 2.6 (1)]{MR1393941} and \eqref{eq:natrad}. Moreover, by \cite[Ch.\ X, Proposition 2.7]{MR1393941}, since $\mathsf{G}$ preserves $J$, we have $\nabla'J=0$. By uniqueness, $\nabla'$ therefore coincides with the Bismut connection $\nabla$ of $(M,J,g)$, which concludes the proof.
\end{proof}

\medskip
\section{The canonical reduction of Bismut--Ambrose--Singer manifolds} 
\label{sect:reductive} \setcounter{equation} 0

In this section, we combine Theorem \ref{thm:BASnatrad} with \cite[Theorem, p.\ 4]{MR0519928} to identify a canonical subgroup of the automorphism group of a BAS manifold $(M,J,g)$ with respect to which $(M,g)$ is naturally reductive and the metric $g$ is induced by a non-degenerate, bi-invariant scalar product. We then use this presentation of $M$ as a Lie group quotient to prove Theorem \ref{thm:MAIN-general}.

\subsection{The canonical presentation of Bismut--Ambrose--Singer manifolds} \label{sect:canpres} \hfill \par

Let $(M,J,g)$ be a complete, simply-connected BAS manifold and fix a point $p \in M$. Then, the Bismut connection $\nabla$ induces a \emph{canonical reductive decomposition}
\begin{equation} \tag{CD} \label{eq:canreddec}
\mathfrak{aut}(M) = \mathfrak{aut}_{0}(M) + \mathfrak{m}
\end{equation}
as in Theorem \ref{thm:BASnatrad}. For simplicity, we suppress the dependence on $p$ in the notation. Recall that one can identify $\mathfrak{m} \simeq T_pM$ via the evaluation map. 

Consider the subspaces
\begin{equation} \label{eq:def-lu}
\mathfrak{l} \coloneqq [\mathfrak{m}, \mathfrak{m}] + \mathfrak{m} \subset \mathfrak{aut}(M) \,\, , \quad
\mathfrak{u} \coloneqq \mathfrak{l} \cap \mathfrak{aut}_{0}(M)
\end{equation}
and notice that, since $[\mathfrak{aut}_{0}(M), \mathfrak{m}] \subset \mathfrak{m}$, $\mathfrak{l}$ is an ideal of $\mathfrak{aut}(M)$, and $\mathfrak{u}$ is a subalgebra. Notice that choosing a different base point leaves $\mathfrak{l}$ unchanged, while the corresponding isotropy subgroup changes by conjugation.

\begin{proposition} \label{prop:Kostant}
Let $(M,J,g)$ be a complete, simply-connected BAS manifold, consider its canonical reductive decomposition \eqref{eq:canreddec} and the Lie algebras \eqref{eq:def-lu}. Then, the following properties hold. \begin{itemize}
\item[$i)$] The Bismut connection $\nabla$ of $(M,J,g)$ coincides with the canonical Ambrose--Singer connection determined by $\mathfrak{m}$.
\item[$ii)$] If a tensor field on $M$ is $\mathrm{Aut}^\mathrm{o}(M)$-invariant, then it is $\nabla$-parallel.
\item[$iii)$] The holonomy algebra of the Bismut connection of $(M,J,g)$ at $p$ coincides with $\mathfrak{u}$.
\item[$iv)$] There exists a unique $\mathrm{ad}(\mathfrak{l})$-invariant, symmetric, bilinear form $Q$ on $\mathfrak{l}$ such that
$$
\text{$Q|_{\mathfrak{u} \otimes \mathfrak{u}}$ is non-degenerate} \,\, , \quad
Q(\mathfrak{u}, \mathfrak{m}) = 0 \,\, , \quad
Q|_{\mathfrak{m} \otimes \mathfrak{m}} = g \,\, .
$$
\end{itemize}
\end{proposition}

\begin{proof}
Claim $i)$ follows directly from the Nomizu construction carried out in Section \ref{sect:Nomcons}. Claim $ii)$ follows from \cite[Proposition 2.7, Ch.\ X]{MR1393941}. Claim $iii)$ follows from \cite[Corollary 4.3, Ch.\ X]{MR1393941}. Finally, claim $iv)$ follows from \cite[Theorem, p.\ 4]{MR0519928}.
\end{proof}

We establish some useful properties of BAS manifolds with respect to their canonical reductive decomposition.

\begin{lemma} \label{lem:LUprop}
Let $(M,J,g)$ be a complete, simply-connected BAS manifold, and consider its canonical reductive decomposition \eqref{eq:canreddec} and its holonomy algebra $\mathfrak{u}$ at a point $p \in M$.
\begin{itemize}
\item[$i)$] The torsion $T$ and the curvature $R$ of the Bismut connection $\nabla$ are given by
\begin{equation} \label{eq:TR}
T(X,Y) = -[X,Y]_{\mathfrak{m}} \quad \text{and} \quad R(X,Y).Z = [[X,Y]_{\mathfrak{u}},Z] \quad \text{for all $X, Y, Z \in \mathfrak{m}$} \,\, .
\end{equation}
\item[$ii)$] The torsion $T^{\mathrm{Ch}}$ of the Chern connection $\nabla^{\mathrm{Ch}}$ is given by
\begin{equation} \label{eq:TCh}
T^{\rm Ch}(X,Y) = -\tfrac12\big([JX,JY]_{\mathfrak{m}} -[X,Y]_{\mathfrak{m}}\big) \quad \text{for all $X, Y \in \mathfrak{m}$} \,\, .
\end{equation}
\item[$iii)$] The value at $p$ of the Lee vector field $\theta^{\sharp}$ is given by
\begin{equation} \label{eq:thetasharp}
\theta^{\sharp}\big|_p = -\sum_{i=1}^n J[e_i,Je_i]_{\mathfrak{m}} \,\, ,
\end{equation}
where $\{e_i,Je_i\}_{i=1,{\dots},n}$ is a $(J,g)$-unitary basis of $\mathfrak{m}$.
\end{itemize}
\end{lemma}

\begin{proof}
Equation \eqref{eq:TR} follows directly from \cite[Theorem 2.6, Ch.\ X]{MR1393941}. Equation \eqref{eq:TCh} follows from \eqref{eq:defTCh}, from the identity
$$
\mathrm{d}\omega(X,Y,Z) = g([X,Y]_{\mathfrak{m}},JZ) +g([Y,Z]_{\mathfrak{m}},JX) +g([Z,X]_{\mathfrak{m}},JY) \quad \text{for all $X, Y, Z \in \mathfrak{m}$} \,\, ,
$$
and from the naturally reductive condition \eqref{eq:natrad}. Finally, equation \eqref{eq:thetasharp} follows from \eqref{def:theta} and \eqref{eq:TCh}.
\end{proof}

From Lemma \ref{lem:LUprop}, one can deduce the following facts.

\begin{lemma}
Let $(M,J,g)$ be a complete, simply-connected BAS manifold.
\begin{itemize}
\item[$i)$] The curvature $R$ of the Bismut connection is of type $(1,1)$, \emph{i.e.},
\begin{equation} \label{eq:R11}
R(JV,JW) = R(V,W) \quad \text{for all $V, W \in \Gamma(TM)$} \,\, .
\end{equation}
\item[$ii)$] The torsion $T^{\rm Ch}$ of the Chern connection satisfies the Jacobi identity, \emph{i.e.},
\begin{multline} \label{eq:Tch-Jac}
T^{\rm Ch}(T^{\rm Ch}(U, V), W) +T^{\rm Ch}(T^{\rm Ch}(W, U), V) \\
+T^{\rm Ch}(T^{\rm Ch}(V, W), U) = 0 \quad \text{for all $U, V, W \in \Gamma(TM)$} \,\, .
\end{multline}
\item[$iii)$] The Lee vector field $\theta^{\sharp}$ satisfies the following properties:
\begin{gather}
g(T^{\rm Ch}(V,W),\theta^{\sharp}) = g(T^{\rm Ch}(V,W),J\theta^{\sharp}) = 0 \quad \text{for all $V, W \in \Gamma(TM)$} \,\, , \label{eq:Ttheta1} \\
JT(\theta^{\sharp},V) = T(\theta^{\sharp},JV) \quad \text{and} \quad JT(J\theta^{\sharp},V) = T(J\theta^{\sharp},JV) \quad \text{for all $V \in \Gamma(TM)$} \,\, . \label{eq:Ttheta2}
\end{gather}
\end{itemize}
\end{lemma}

\begin{proof}
Consider the canonical reductive decomposition \eqref{eq:canreddec} of $(M,J,g)$ and its holonomy algebra $\mathfrak{u}$ at a point $p \in M$. Let also $Q$ be as in Proposition \ref{prop:Kostant}. Since $Q|_{\mathfrak{u} \otimes \mathfrak{u}}$ is non-degenerate, a straightforward computation shows that
$$
[JX,JY]_{\mathfrak{u}} = [X,Y]_{\mathfrak{u}} \quad \text{for all $X,Y \in \mathfrak{m}$}
$$
and so equation \eqref{eq:R11} follows from \eqref{eq:TR}. To prove the second claim, we observe that by \eqref{eq:TCh} and by the integrability of $J$, it follows that
\begin{multline} \label{eq:TcTc}
4\,T^{\rm Ch}(T^{\rm Ch}(X, Y), Z) = 
-[[JX,Y]_{\mathfrak{m}},JZ]_{\mathfrak{m}}
-[[X,JY]_{\mathfrak{m}},JZ]_{\mathfrak{m}} \\
-[[JX,JY]_{\mathfrak{m}},Z]_{\mathfrak{m}}
+[[X,Y]_{\mathfrak{m}},Z]_{\mathfrak{m}}
\quad \text{for all $X, Y, Z \in \mathfrak{m}$} \,\, .
\end{multline}
Therefore, by \eqref{eq:TR}, \eqref{eq:R11}, \eqref{eq:TcTc} and the Jacobi identity, we obtain for every $X, Y, Z \in \mathfrak{m}$
$$\begin{aligned}
4\underset{X, Y, Z}{\mathfrak{S}} & T^{\rm Ch}(T^{\rm Ch}(X, Y), Z) = \\ 
& =- \underset{X, Y, Z}{\mathfrak{S}} \Big(
[[JX,Y]_{\mathfrak{m}},JZ]_{\mathfrak{m}}
+[[X,JY]_{\mathfrak{m}},JZ]_{\mathfrak{m}}
+[[JX,JY]_{\mathfrak{m}},Z]_{\mathfrak{m}}
-[[X,Y]_{\mathfrak{m}},Z]_{\mathfrak{m}} \Big) \\
& = \underset{X, Y, Z}{\mathfrak{S}} \Big(
[[JX,Y]_{\mathfrak{u}},JZ]
+[[X,JY]_{\mathfrak{u}},JZ]
+[[JX,JY]_{\mathfrak{u}},Z]
-[[X,Y]_{\mathfrak{u}},Z] \Big) \\
& = \underset{X, Y, Z}{\mathfrak{S}} \Big(
R(JX,Y).JZ
+R(X,JY).JZ
+R(JX,JY).Z
-R(X,Y).Z \Big) \\
& = 0 \,\, ,
\end{aligned}$$
and hence equation \eqref{eq:Tch-Jac} follows. To prove the third claim, we observe that equation \eqref{eq:Ttheta1} follows by tracing the identity
$$
g(T^{\rm Ch}(T^{\rm Ch}(X, Y), Z_1),Z_2) +g(T^{\rm Ch}(T^{\rm Ch}(Y, Z_1), X),Z_2) +g(T^{\rm Ch}(T^{\rm Ch}(Z_1, X), Y),Z_2) = 0
$$
with respect to $Z_1$ and $Z_2$ and by \eqref{def:theta}. Finally, equation \eqref{eq:Ttheta2} follows by \eqref{eq:TR}, \eqref{eq:TCh} and \eqref{eq:Ttheta1}.
\end{proof}

Finally, we infer the following properties of the Lee vector field.

\begin{proposition} \label{prop:thetaholkill}
Let $(M,J,g)$ be a complete, simply-connected BAS manifold and let $\theta^{\sharp}$ be its Lee vector field. Then, both $\theta^{\sharp}$ and $J\theta^{\sharp}$ are $\mathrm{Aut}^\mathrm{o}(M)$-invariant, holomorphic Killing vector fields.
\end{proposition}

\begin{proof}
The first claim follows from the fact that the Lee form $\theta$ is $\mathrm{Aut}^\mathrm{o}(M)$-invariant, and so its metric dual $\theta^{\sharp}$ is $\mathrm{Aut}^\mathrm{o}(M)$-invariant as well. By Proposition \ref{prop:Kostant}, it follows that $\nabla \theta^{\sharp} = 0$ . Since the Bismut connection $\nabla$ is metric and its torsion $T$ is totally skew-symmetric, it follows that
$$
(\mathcal{L}_{\theta^{\sharp}}g)(V,W) = g(\nabla_{V}\theta^{\sharp},W) +g(\nabla_{W}\theta^{\sharp},V) \quad \text{for all $V, W \in \Gamma(TM)$}
$$
and so $\theta^{\sharp}$ is a Killing vector field. To show that $\theta^{\sharp}$ is holomorphic, we observe that for every $V \in \Gamma(TM)$ 
\begin{align*}
(\mathcal{L}_{\theta^{\sharp}}J)V &= [\theta^{\sharp},JV] - J[\theta^{\sharp},V] \\
&= T(JV,\theta^{\sharp}) + \nabla_{\theta^{\sharp}}JV - \nabla_{JV}\theta^{\sharp} - JT(V,\theta^{\sharp}) -J\nabla_{\theta^{\sharp}}V +J\nabla_V\theta^{\sharp} \\
&= T(JV,\theta^{\sharp}) - JT(V,\theta^{\sharp})
\end{align*}
and so $\mathcal{L}_{\theta^{\sharp}}J = 0$ by \eqref{eq:Ttheta2}. The same argument applies to $J\theta^{\sharp}$, and so this concludes the proof.
\end{proof}

\subsection{The canonical reduction theorem} \label{sect:proofA} \hfill \par

The aim of this section is to prove Theorem \ref{thm:MAIN-general}. We first introduce a distinguished class of BAS manifolds that arises naturally in this setting.

\begin{definition} \label{def:reduced}
A complete, simply-connected BAS manifold $(M,J,g)$ is said to be \emph{reduced} if it admits no non-trivial $\mathrm{Aut}^\mathrm{o}(M)$-invariant holomorphic Killing vector fields. 
\end{definition}

As a straightforward consequence of Proposition \ref{prop:thetaholkill}, we obtain the following observation.

\begin{corollary}
Every complete, simply-connected, reduced BAS manifold is balanced.
\end{corollary}

We provide the following structural result for BAS manifolds from which Theorem \ref{thm:MAIN-general} follows.

\begin{theorem} \label{thm:generalstructure}
Let $(M,J,g)$ be a complete, simply-connected BAS manifold, and let $\mathsf{Z}(M)$ be the connected component of the centre of $\mathrm{Aut}^\mathrm{o}(M)$. Then, the projection map
\begin{equation} \label{eq:Tits}
\pi: M \to B \coloneqq M/\mathsf{Z}(M)
\end{equation}
is an abelian principal bundle, and the base $B$ is simply-connected. Moreover, there exists a uniquely determined Hermitian structure $(\check{J}, \check{g})$ on $B$ such that $(B,\check{J},\check{g})$ is a reduced BAS manifold and $\pi$ is a holomorphic Riemannian submersion with totally geodesic fibres.
\end{theorem}

\begin{proof}
Consider the space
\begin{equation}
\mathfrak{z}(M) \coloneqq \{V \in \mathfrak{aut}(M) : \text{$V$ is $\mathrm{Aut}^\mathrm{o}(M)$-invariant}\} \,\, .
\end{equation}
By definition, it follows that $\mathfrak{z}(M)$ is the centre of $\mathfrak{aut}(M)$. Moreover, by Proposition \ref{prop:Kostant}, it follows that every $V \in \mathfrak{z}(M)$ is $\nabla$-parallel. This implies then that $\mathfrak{z}(M)$ is $J$-invariant. By construction, $\mathsf{Z}(M)$ is the connected Lie subgroup of $\mathrm{Aut}^\mathrm{o}(M)$ with Lie algebra $\mathfrak{z}(M)$. The action of $\mathsf{Z}(M)$ is holomorphic, isometric and free. Hence, the quotient $B = M/\mathsf{Z}(M)$ is a smooth manifold, the projection \eqref{eq:Tits} is an abelian principal bundle, and there exists a unique Hermitian structure $(\check{J},\check{g})$ on $B$ with respect to which $\pi$ is a holomorphic Riemannian submersion. Since $\mathfrak{z}(M)$ is central, it follows that
$$
g([V_1,V_2],X) +g([V_1,X],V_2) +g([V_2,X],V_1) = 0 \quad \text{for every $V_1, V_2 \in \mathfrak{z}(M)$ and $X \in \mathfrak{aut}(M)$} \,\, ,
$$
and so the fibres of $\pi$ are totally geodesic.

To show that $(B,\check{J},\check{g})$ is BAS, consider the canonical reductive decomposition \eqref{eq:canreddec} of $(M,J,g)$ and denote by $\mathrm{pr}_{\mathfrak{m}} : \mathfrak{aut}(M) \to \mathfrak{m}$ the corresponding projection over $\mathfrak{m}$. Since $\mathfrak{aut}_0(M) \cap \mathfrak{z}(M) = \{0\}$, the restriction $\mathrm{pr}_{\mathfrak{m}}|_{\mathfrak{z}(M)} : \mathfrak{z}(M) \to \mathfrak{m}$ is injective. Since the canonical decomposition is $\mathrm{ad}(\mathfrak{aut}_0(M))$-invariant, it follows that
$$
\mathfrak{f} \coloneqq \mathrm{pr}_{\mathfrak{m}}(\mathfrak{z}(M)) \subset \mathfrak{m}
$$
is a trivial $\mathrm{ad}(\mathfrak{aut}_0(M))$-submodule of $\mathfrak{m}$. Moreover, for every $U,V \in \mathfrak{z}(M)$, we have
$$
[\mathrm{pr}_{\mathfrak{m}}(U),\mathrm{pr}_{\mathfrak{m}}(V)] = -[U -\mathrm{pr}_{\mathfrak{m}}(U) ,\mathrm{pr}_{\mathfrak{m}}(V)] = -\mathrm{pr}_{\mathfrak{m}}([U -\mathrm{pr}_{\mathfrak{m}}(U),V]) = 0
$$
and so $\mathfrak{f}$ is an abelian subalgebra. Therefore, the naturally-reductive condition \eqref{eq:natrad} implies that the $g$-orthogonal complement $\mathfrak{b}$ of $\mathfrak{f}$ in $\mathfrak{m}$ is a reductive complement for the Lie subalgebra $\mathfrak{aut}_0(M) +\mathfrak{f}$ in $\mathfrak{aut}(M)$. Moreover, the naturally-reductive condition \eqref{eq:natrad} for $\mathfrak{m}$ implies that
$$
g([X,Y]_{\mathfrak{b}},Z) +g(Y,[X,Z]_{\mathfrak{b}}) = 0 \quad \text{for all $X,Y,Z \in \mathfrak{b}$} \,\, ,
$$
which proves that $(B,\check{J},\check{g})$ is BAS.

To prove that $(B,\check{J},\check{g})$ is reduced, we consider the principal abelian bundle $\pi : M \to B$, we denote by $\mathcal{V}$ the vertical distribution, spanned by the elements in $\mathfrak{z}(M)$, and by $\mathcal{H}$ the orthogonal distribution. Recall that a horizontal vector field $X \in \Gamma(\mathcal{H})$ is projectable to $B$ if and only if it is $\mathsf{Z}(M)$-invariant. On the other hand, every vector field on $B$ admits a unique horizontal, $\mathsf{Z}(M)$-invariant vector field on $M$. We also recall that
$$
[X,U] = 0 \quad \text{for all $X \in \Gamma(\mathcal{H})^{\mathsf{Z}(M)}$, $U \in \mathfrak{z}(M)$} \,\, .
$$
Let now $X \in \Gamma(\mathcal{H})^{\mathsf{Z}(M)}$ and assume that $\pi_*X \in \mathfrak{z}(B)$. Since $\mathrm{Aut}^\mathrm{o}(M)$ acts on $B$ by holomorphic isometries which are isotopic to the identity, it follows that $\pi_*X$ is $\mathrm{Aut}^\mathrm{o}(M)$-invariant. Since $\pi$ is $\mathrm{Aut}^\mathrm{o}(M)$-equivariant, it follows that $X$ is $\mathrm{Aut}^\mathrm{o}(M)$-invariant. By Proposition \ref{prop:Kostant}, it follows that $\nabla X = 0$, and so $X$ is Killing. Finally, to show that $X$ is also holomorphic, we observe that for all $Y, Z \in \Gamma(\mathcal{H})^{\mathsf{Z}(M)}$ and $U \in \mathfrak{z}(M)$, we have
$$\begin{aligned}
(\mathcal{L}_XJ)(U) &= [X,JU] -J[X,U] = 0 \,\, , \\
g((\mathcal{L}_XJ)(Y),U) &= \mathcal{L}_X(g(JY,U)) -g(JY,[X,U])
+\mathcal{L}_X(g(Y,JU)) -g(Y,[X,JU]) = 0 \,\, , \\
g((\mathcal{L}_XJ)(Y),Z) &= \check{g}((\mathcal{L}_{\pi_*X}\check{J})(\pi_*Y),\pi_*Z) \circ \pi = 0 \,\, .
\end{aligned}$$
It then follows that $X \in \mathfrak{z}(M)$, and so $\pi_*X = 0$.
\end{proof}

\medskip
\section{Compact Bismut--Ambrose--Singer manifolds with finite fundamental group} 
\label{sect:compact} \setcounter{equation} 0

The aim of this section is to prove Theorem \ref{thm:MAIN-cpt}, namely the classification of compact BAS manifolds with finite fundamental group. Up to passing to a finite cover, this amounts to considering the simply-connected case. 

We begin by establishing some fundamental concepts regarding \emph{Wang C-spaces} \cite{MR0066011}, as they provide the underlying geometric framework for simply-connected, compact BAS manifolds.
We first recall that a \emph{flag manifold} is an adjoint orbit $\mathsf{G}/\mathsf{K}$ of a compact semisimple Lie group $\mathsf{G}$. Equivalently, the isotropy subgroup $\mathsf{K}$ is the centralizer of a torus in $\mathsf{G}$. 
A Wang C-space is the total space of a $\mathsf{G}$-equivariant principal torus bundle $\mathsf{T}=\mathsf{K}/\mathsf{H} \rightarrow \mathsf{G/H} \rightarrow \mathsf{G}/\mathsf{K}$ over a flag manifold, endowed with a left-$\mathsf{G}$ and right-$\mathsf{T}$ invariant complex structure. The projection map $\mathsf{G/H} \rightarrow \mathsf{G}/\mathsf{K}$ is called \emph{Tits fibration}. A Riemannian metric on $\mathsf{G}/\mathsf{H}$ (resp.\ $\mathsf{G}/\mathsf{K}$) is called \emph{normal} if it is induced by a positive definite, $\mathrm{Ad}(\mathsf{G})$-invariant scalar product on the Lie algebra $\mathfrak{g}$ of $\mathsf{G}$. We provide here a proof of the following splitting result, which is well-known to the experts (see, \emph{e.g.}, \cite[p.\ 194]{MR1922121}).

\begin{lemma} \label{lem:dRflags}
Let $(\mathsf{G}/\mathsf{K},\check{g})$ be a flag manifold endowed with an invariant Riemannian metric, where $\mathsf{G}$ is a compact, semisimple, simply-connected Lie group. Denote by $\mathsf{G} = \mathsf{G}_1 \times {\dots} \times \mathsf{G}_s$ the decomposition of $\mathsf{G}$ into simple factors and by $\mathrm{pr}_i : \mathsf{G} \to \mathsf{G}_i$ the canonical projection onto the $i$-th factor. Then
$$
\mathsf{K} = \mathsf{K}_1 \times {\dots} \times \mathsf{K}_s \,\, , \quad
\text{where} \quad \mathsf{K}_i \coloneqq \mathsf{G}_i \cap \mathsf{K} \quad \text{for $1 \leq i \leq s$} \,\, .
$$
Moreover, for every $1 \leq i \leq s$, the quotient $\mathsf{G}_i/\mathsf{K}_i$ is a flag manifold and there exists a uniquely determined $\mathsf{G}_i$-invariant metric $\check{g}_i$ on $\mathsf{G}_i/\mathsf{K}_i$ such that the map
$$
(\mathsf{G}/\mathsf{K}, \check{g}) \to (\mathsf{G}_1/\mathsf{K}_1, \check{g}_1) \times {\dots} \times (\mathsf{G}_s/\mathsf{K}_s, \check{g}_s) \,\, , \quad
a\mathsf{K} \mapsto \big(\mathrm{pr}_1(a)\mathsf{K}_1,{\dots},\mathrm{pr}_s(a)\mathsf{K}_s\big)
$$
is an isometry.
\end{lemma}

\begin{proof}
Since $\mathsf{G}/\mathsf{K}$ is a flag manifold, there exists an element $k \in \mathsf{G}$ such that $\mathsf{K}$ is the centralizer of $k$, \emph{i.e.}, $\mathsf{K} = C_{\mathsf{G}}(k)$. Moreover
$$
C_{\mathsf{G}}(k) = C_{\mathsf{G}_1 \times {\dots} \times \mathsf{G}_s}\big(\mathrm{pr}_1(k),{\dots},\mathrm{pr}_s(k)\big) = C_{\mathsf{G}_1}\big(\mathrm{pr}_1(k)\big) \times {\dots} \times C_{\mathsf{G}_s}\big(\mathrm{pr}_s(k)\big) \,\, ,
$$
and
$$
C_{\mathsf{G}_i}\big(\mathrm{pr}_i(k)\big) = C_{\mathsf{G}}(k) \cap \mathsf{G}_i \,\, .
$$
Therefore, letting $\mathsf{K}_i \coloneqq C_{\mathsf{G}_i}(k_i)$, we obtain a well-defined diffeomorphism
\begin{equation} \label{eq:decG/K}
\mathsf{G}/\mathsf{K} \to \mathsf{G}_1/\mathsf{K}_1 \times {\dots} \times \mathsf{G}_s/\mathsf{K}_s \,\, , \quad
a\mathsf{K} \mapsto \big(\mathrm{pr}_1(a)\mathsf{K}_1,{\dots},\mathrm{pr}_s(a)\mathsf{K}_s\big) \,\, .
\end{equation}
By construction, it follows that $\mathrm{rank}(\mathsf{K}_i) = \mathrm{rank}(\mathsf{G}_i)$ for every $1 \leq i \leq s$. Hence, by Schur's Lemma, the metric $\check{g}$ splits according to the decomposition \eqref{eq:decG/K}, and defines a $\mathsf{G}_i$-invariant Riemannian metric $\check{g}_i$ on $\mathsf{G}_i/\mathsf{K}_i$. This concludes the proof.
\end{proof}

We now prove the following two results, namely Theorem \ref{thm:compact-structure} and Proposition \ref{prop:cmpctBAS}, from which Theorem \ref{thm:MAIN-cpt} follows directly.

\begin{theorem} \label{thm:compact-structure}
Let $(M,J,g)$ be a compact, simply-connected BAS manifold. Then, $(M,J,g)$ is a Wang C-space and its canonical reduction coincides with the Tits fibration over the corresponding flag manifold endowed with a normal metric.
\end{theorem}

\begin{proof}
Let $(M,J,g)$ be a compact, simply-connected BAS manifold, and fix $p \in M$. Since $M$ is compact, it follows that $\mathrm{Aut}^{\mathrm{o}}(M) = \mathsf{G} \cdot \mathsf{T}$ for some compact semisimple Lie group $\mathsf{G}$ and a torus $\mathsf{T}$. By \cite[Corollary 1]{MR0037311}, we can assume that $\mathsf{G}$ acts transitively on $M$, and we denote by $\mathsf{H}$ its isotropy at $p$. Consider the canonical reduction
$$
\pi: (M,J,g) \to (B,\check{J},\check{g})
$$
as described in Theorem \ref{thm:generalstructure}, with the associated canonical reductive decomposition
$$
\mathfrak{aut}(M) = \mathfrak{aut}_0(M) + \mathfrak{m} = \mathfrak{g} + \mathfrak{t} \,\, , \quad
\mathfrak{m} = \mathfrak{f} +\mathfrak{b} \,\, ,
$$
where $\mathfrak{g}$ is the Lie algebra of $\mathsf{G}$, $\mathfrak{t}$ is the Lie algebra of $\mathsf{T}$, and $\mathfrak{f}$ is the projection of $\mathfrak{z}(M)$ on $\mathfrak{m}$. We then consider the trivial module
$$
\mathfrak{m}_0 \coloneqq \left\{ X\in\mathfrak{m} : [\mathfrak{aut}_0(M),X] = 0 \right\} \,\, .
$$
We claim that $\mathfrak{m}_0 = \mathfrak{f}$. Indeed, $\mathfrak{m}_0$ is a subalgebra of $\mathfrak{aut}(M)$ and can be identified with the Lie algebra of the compact group
$$
\mathsf{E} = N_{\mathrm{Aut}^\mathrm{o}(M)}(\mathrm{Aut}_0(M))^\mathrm{o} \big/ \mathrm{Aut}_0(M) \,\, .
$$
obtained by taking the quotient of the connected component of the normaliser of the isotropy subgroup $\mathrm{Aut}_0(M)$ in $\mathrm{Aut}^\mathrm{o}(M)$ by $\mathrm{Aut}_0(M)$. By the naturally reductive condition \eqref{eq:natrad}, the right action of $\mathsf{E}$ on $M$ is isometric. By \cite{MR0066011} (see also \cite[Theorem 1.5]{MR5010304}), there exists a maximal torus $\tilde{\mathsf{T}} \subset \mathsf{E}$ whose right action preserves $J$. By maximality of $\mathrm{Aut}^\mathrm{o}(M)$ and by the fact that the right-$\tilde{\mathsf{T}}$ action commutes with the left-$\mathrm{Aut}^\mathrm{o}(M)$ action, it follows that $\tilde{\mathsf{T}} \subset \mathsf{Z}(M)$. It then follows that $\tilde{\mathsf{T}}$ commutes with the whole $\mathsf{E}$. Since a compact Lie group with a central maximal torus is abelian, it follows that $\mathsf{E} = \tilde{\mathsf{T}}$. This shows in particular that $\mathfrak{m}_0 \subset \mathfrak{f}$. Since we have already proved in the proof of Theorem \ref{thm:generalstructure} that $\mathfrak{f} \subset \mathfrak{m}_0$, this proves the equality.

By the argument above, it follows that $\mathfrak{b}$ does not contain any trivial $\mathrm{ad}(\mathfrak{aut}_0(M))$-module. Since the isotropy $\mathrm{Aut}_0(M)$ is compact, it then follows that $\mathfrak{b} = [\mathfrak{aut}_0(M),\mathfrak{b}]$ and so
$$
\mathfrak{b} = [\mathfrak{aut}_0(M),\mathfrak{b}] \subset [\mathfrak{aut}(M),\mathfrak{aut}(M)] = \mathfrak{g} \,\, .
$$
In particular, the left-$\mathsf{G}$ action on $M$ descends via $\pi$ to a transitive left-$\mathsf{G}$ action on $(B,\check{J},\check{g})$ by holomorphic isometries. We denote by $\mathsf{K}$ the isotropy of $\mathsf{G}$ at $\pi(p)$, and we observe that $\mathsf{K}$ is connected and that its Lie algebra $\mathfrak{k}$ satisfies
$$
\mathfrak{k} = (\mathfrak{aut}_0(M) + \mathfrak{f}) \cap \mathfrak{g} \,\, .
$$
Moreover, since $\pi : M \to B$ is a principal $\mathsf{Z}(M)$-bundle, it follows that $\mathsf{H}$ is a normal subgroup of $\mathsf{K}$ and $\mathsf{K}/\mathsf{H} \simeq \mathsf{Z}(M)$. The $\mathsf{G}$ action on $M$ is clearly effective. We claim then that the $\mathsf{G}$ action on $B$ is almost-effective. Indeed, let $\mathfrak{i}$ be an ideal of $\mathfrak{g}$ contained in $\mathfrak{k}$. Then, it is straightforward to check that $\mathfrak{i} \cap \mathfrak{h}$ is an ideal of $\mathfrak{g}$ contained in $\mathfrak{h}$, where $\mathfrak{h}$ denotes the Lie algebra of $\mathsf{H}$. This shows that $\mathfrak{i} \cap \mathfrak{h} = \{0\}$. Since $[\mathfrak{i}, \mathfrak{b}] = \mathfrak{i} \cap \mathfrak{b} = \{0\}$ and $[\mathfrak{k},\mathfrak{k}] \subset \mathfrak{h}$, it follows that $[\mathfrak{i}, \mathfrak{g}] \subset \mathfrak{i} \cap \mathfrak{h} = \{0\}$, hence $\mathfrak{i} \subset \mathfrak{z}(\mathfrak{g})$. Since $\mathfrak{g}$ is semisimple, we obtain $\mathfrak{i} = \{0\}$.

By Theorem \ref{thm:generalstructure}, it follows that $(B,\check{J},\check{g})$ is $(\mathsf{G},\pi(p),\mathfrak{b})$-naturally reductive and reduced. Moreover, the isotropy representation of $B = \mathsf{G}/\mathsf{K}$ has no trivial submodules. Therefore, by \cite{MR0066011}, it follows that $B = \mathsf{G}/\mathsf{K}$ is a flag manifold. It remains to show that the metric $\check{g}$ is normal.

By \cite[Ch.\ X, Theorem 5.2]{MR1393941}, it follows that $(B,\check{g})$ splits isometrically as
$$
(B,\check{g}) = (\mathsf{G}_1/\mathsf{K}_1, \check{g}_1) \times {\dots} \times (\mathsf{G}_s/\mathsf{K}_s, \check{g}_s) \,\, ,
$$
where $(\mathsf{G}_i/\mathsf{K}_i, \check{g}_i)$ is de Rham irreducible and naturally reductive for all $1 \leq i \leq s$, with $\mathsf{G} = \mathsf{G}_1 \times {\dots} \times \mathsf{G}_s$ and $\mathsf{K}_i = \mathsf{G}_i \cap \mathsf{K}$. It is routine to check that $(\mathsf{G}_i/\mathsf{K}_i, \check{g}_i)$ is a flag manifold for all $1 \leq i \leq s$, and so $\mathsf{G}_i$ is simple by Lemma \ref{lem:dRflags}. Consider the corresponding splitting at the Lie algebra level
$$
\mathfrak{g} = \mathfrak{g}_1 + \ldots + \mathfrak{g}_s \,\, , \quad
\mathfrak{k} = \mathfrak{k}_1 + \ldots + \mathfrak{k}_s \,\, , \quad
\mathfrak{b} = \mathfrak{b}_1 + \ldots + \mathfrak{b}_s \,\, ,
$$
with $\mathfrak{g}_i = \mathfrak{k}_i + \mathfrak{b}_i$ for all $1 \leq i \leq s$. By Schur's Lemma, the complex structure $\check{J}$ splits according to the decomposition of $\mathfrak{b}$, hence it defines an orthogonal complex structure $\check{J}_i$ on $(\mathsf{G}_i/\mathsf{K}_i,\check{g}_i)$.

Since $\mathfrak{b}_i +[\mathfrak{b}_i,\mathfrak{b}_i]$ is an ideal of $\mathfrak{g}_i$ and $\mathfrak{g}_i$ is simple, it follows that $\mathfrak{g}_i = \mathfrak{b}_i +[\mathfrak{b}_i,\mathfrak{b}_i]$. Therefore, by Kostant's Theorem (see \cite[p.\ 4]{MR0519928}), there exists a unique ${\rm Ad}(\mathsf{G}_i)$-invariant, symmetric, non-degenerate, bilinear form $Q_i$ on $\mathfrak{g}_i$ such that $Q_i(\mathfrak{k}_i,\mathfrak{b}_i) = 0$ and $Q_i|_{\mathfrak{b}_i \otimes \mathfrak{b}_i} = \check{g}_i$. Since $\mathsf{G}_i$ is simple, it follows that $Q_i$ is positive definite. Therefore, $\check{g}_i$ is normal for all $1 \leq i \leq s$, hence $\check{g}$ is normal.
\end{proof}

Although the metric $\check{g}$ induced on the base $(B,\check{J})$ is normal by \ref{thm:compact-structure}, the metric $g$ on the total space $(M,J)$ need not be normal. Consider for example the Calabi--Eckmann threefold $M = S^3 \times S^3$ with its standard complex structure $J$. Then, $M$ is the total space of a holomorphic principal $\mathsf{T}^2$-bundle $\pi: M \to \mathbb{CP}^1 \times \mathbb{CP}^1$, and every left-$(\mathsf{SU}(2)\times\mathsf{SU}(2))$ and right-$\mathsf{T}^2$ invariant Hermitian metric $g$ that projects onto the standard metric on $\mathbb{CP}^1 \times \mathbb{CP}^1$ turns $(M,J,g)$ into a BAS manifold. More generally, the following result holds.

\begin{proposition} \label{prop:cmpctBAS}
Let $\pi: \mathsf{G}/\mathsf{H} \to \mathsf{G}/\mathsf{H}\mathsf{T}$ be a homogeneous principal torus bundle over a flag manifold, where $\mathsf{G}$ is a compact semisimple Lie group acting almost-effectively on $\mathsf{G}/\mathsf{H}$, and let $\mathsf{A}$ be a connected abelian Lie group. Assume that $\mathsf{A} \times \mathsf{G}/\mathsf{H}$ admits a left-$(\mathsf{A} \times \mathsf{G})$ and right-$(\mathsf{A} \times \mathsf{T})$-invariant Hermitian structure $(J,g)$ and that $g$ induces on $\mathsf{G}/\mathsf{H}\mathsf{T}$ a normal metric. Then, $(\mathsf{A} \times\mathsf{G}/\mathsf{H},J,g)$ is BAS.
\end{proposition}

\begin{proof}
We denote by $\mathfrak{g}$, $\mathfrak{h}$, $\mathfrak{t}$, $\mathfrak{a}$ the Lie algebras of $\mathsf{G}$, $\mathsf{H}$, $\mathsf{T}$ and $\mathsf{A}$, respectively. For the sake of notation, we set $\mathfrak{f} \coloneqq \mathfrak{a} \oplus \mathfrak{t}$ and $\tilde{\mathfrak{g}} \coloneqq \mathfrak{a} \oplus \mathfrak{g}$. Let $Q$ be an $\mathrm{Ad}(\mathsf{A} \times \mathsf{G})$-invariant metric on $\tilde{\mathfrak{g}}$ such that, for the corresponding $Q$-orthogonal Lie algebra decomposition
$$
\tilde{\mathfrak{g}} = \mathfrak{h} + \mathfrak{f} + \mathfrak{b} \,\, ,
$$
one has $Q|_{\mathfrak{b} \otimes \mathfrak{b}} = g|_{\mathfrak{b} \otimes \mathfrak{b}}$. Notice that the restriction $Q|_{\mathfrak{g} \otimes \mathfrak{g}}$ is uniquely determined by this condition. By Schur's Lemma, $g(\mathfrak{f},\mathfrak{b}) = 0$ and $J\mathfrak{f} = \mathfrak{f}$. Let us consider the endomorphism $S \in \mathrm{Sym}_+(\mathfrak{f},Q)$ defined by
$$
g(U,V) = Q(SU,V) \quad \text{for all $U,V \in \mathfrak{f}$} \,\, .
$$
To show that $(\mathsf{A} \times \mathsf{G}/\mathsf{H},J,g)$ is BAS, we enlarge the group of holomorphic isometries and find a suitable reductive complement satisfying the naturally reductive condition. To this end, let $\mathfrak{f}_1$ denote the (possibly trivial) eigenspace of $S$ corresponding to the eigenvalue $1$, and consider the orthogonal decomposition $\mathfrak{f} = \mathfrak{f}_1 + \mathfrak{f}_1^{\perp}$. Define then the Lie algebras
$$
\hat{\mathfrak{g}} \coloneqq \tilde{\mathfrak{g}} \oplus \mathfrak{f}_1^{\perp} \,\, , \quad \hat{\mathfrak{h}} \coloneqq \big\{(H+U,U) : H \in \mathfrak{h} \,\, , U \in \mathfrak{f}_1^{\perp}\big\}
$$
and
$$
\mathfrak{m} \coloneqq \mathfrak{m}_1 + \mathfrak{m}_2 \,\, , \quad
\mathfrak{m}_1 \coloneqq \big\{(SU,(S-\mathrm{Id})U) : U \in \mathfrak{f}_1^{\perp} \} \,\, , \quad
\mathfrak{m}_2 \coloneqq \{(X,0) : X \in \mathfrak{f}_1+\mathfrak{b}\} \,\, .
$$
Then, it is routine to check that
$$
\hat{\mathfrak{g}} = \hat{\mathfrak{h}} +\mathfrak{m} \,\, , \quad
\hat{\mathfrak{h}} \cap \mathfrak{m}_1 = \hat{\mathfrak{h}} \cap \mathfrak{m}_2 = \mathfrak{m}_1 \cap \mathfrak{m}_2 = \{0\} \,\, , \quad
[\hat{\mathfrak{h}},\mathfrak{m}_1] = \{0\} \,\, , \quad
[\hat{\mathfrak{h}},\mathfrak{m}_2] \subset \mathfrak{m}_2 \,\, .
$$
We denote by $\widehat{\mathsf{G}}$ the simply-connected Lie group with Lie algebra $\hat{\mathfrak{g}}$ and we observe that, by construction, $\widehat{\mathsf{G}}$ acts transitively on $(\mathsf{A} \times \mathsf{G}/\mathsf{H},J,g)$ by holomorphic isometries. The evaluation map at the origin of $\mathsf{A} \times \mathsf{G}/\mathsf{H}$ induces a linear isomorphism $\mathrm{ev}: \mathfrak{m} \to \mathfrak{f} +\mathfrak{b}$ such that
$$
\mathrm{ev}|_{\mathfrak{m}_1} : (SU,(S-\mathrm{Id})U) \mapsto U \,\, , \quad
\mathrm{ev}|_{\mathfrak{m}_2} : (X,0) \mapsto X \,\, .
$$
We also introduce the $\mathrm{ad}(\hat{\mathfrak{g}})$-invariant bilinear form $\widehat{Q}$ on $\hat{\mathfrak{g}}$ defined by
$$
\widehat{Q}\big((E_1,U_1),(E_2,U_2)\big) \coloneqq Q(E_1,E_2) -Q(S(S-\mathrm{Id})^{-1}U_1,U_2) \,\, .
$$
Then, it is straightforward to check that
$$\begin{gathered}
\widehat{Q}(\hat{\mathfrak{h}},\mathfrak{m}_1) = \widehat{Q}(\hat{\mathfrak{h}},\mathfrak{m}_2) = \widehat{Q}(\mathfrak{m}_1,\mathfrak{m}_2) = 0 \,\, , \\
\widehat{Q}((X_1,0),(X_2,0)) = g(X_1,X_2) \,\, , \quad
\widehat{Q}((SU_1,(S-\mathrm{Id})U_1),(SU_2,(S-\mathrm{Id})U_2)) = g(U_1,U_2)
\end{gathered}$$
for all $U_1,U_2 \in \mathfrak{f}_1^{\perp}$ and $X_1, X_2 \in \mathfrak{f}_1+\mathfrak{b}$. Moreover
$$
\widehat{Q}((H_1+U_1,U_1),(H_2+U_2,U_2)) = Q(H_1,H_2)-Q\big((S-\mathrm{Id})^{-1}U_1,U_2\big)
$$
for all $H_1, H_2 \in \mathfrak{h}$, and so $\widehat{Q}|_{\hat{\mathfrak{h}} \otimes \hat{\mathfrak{h}}}$ is non-degenerate. This shows that $(\mathsf{A} \times \mathsf{G}/\mathsf{H},g)$ is $\widehat{\mathsf{G}}$-naturally reductive, and so $(\mathsf{A} \times \mathsf{G}/\mathsf{H},J,g)$ is BAS by Theorem \ref{thm:BASnatrad}.
\end{proof}

\medskip
\section{Bismut--Ambrose--Singer manifolds of non-compact, semisimple Lie groups} 
\label{sect:noncompact} \setcounter{equation} 0

The aim of this section is to prove Theorem \ref{thm:MAIN-noncpt}, namely the classification of complete, simply-connected BAS manifolds admitting a transitive holomorphic isometric action of a semisimple Lie group of \emph{non-compact type}. Here, a semisimple Lie group is of non-compact type if each simple ideal of its Lie algebra is the Lie algebra of a non-compact simple Lie group.

We recall that every real simple Lie algebra $\mathfrak{g}$ admits a \emph{Cartan decomposition}
$$
\mathfrak{g} = \mathfrak{k} +\mathfrak{p} \,\, , \quad
\text{with} \quad
[\mathfrak{k},\mathfrak{k}] \subset \mathfrak{k} \,\, , \quad
[\mathfrak{k},\mathfrak{p}] \subset \mathfrak{p} \,\, , \quad
[\mathfrak{p},\mathfrak{p}] \subset \mathfrak{k} \,\, .
$$
Here, $\mathfrak{k}$ denotes a maximal compact Lie subalgebra of $\mathfrak{g}$ and $\mathfrak{p}$ is its orthogonal complement with respect to the Cartan-Killing form $\mathscr{B}_{\mathfrak{g}}$ of $\mathfrak{g}$. We also recall that $\mathscr{B}_{\mathfrak{g}}|_{\mathfrak{k} \otimes \mathfrak{k}} < 0$ and $\mathscr{B}_{\mathfrak{g}}|_{\mathfrak{p} \otimes \mathfrak{p}} > 0$, hence
\begin{equation} \label{eq:canmet}
g \coloneqq -\mathscr{B}_{\mathfrak{g}}|_{\mathfrak{k} \otimes \mathfrak{k}} \oplus \mathscr{B}_{\mathfrak{g}}|_{\mathfrak{p} \otimes \mathfrak{p}}
\end{equation}
is a positive definite, $\mathrm{ad}(\mathfrak{k})$-invariant scalar product on $\mathfrak{g}$, called \emph{canonical metric}.

A real simple Lie algebra $\mathfrak{g}$ is called \emph{absolutely simple} if its complexification $\mathfrak{g}^{\mathbb{C}} = \mathfrak{g} \otimes_{\mathbb{R}} \mathbb{C}$ is a complex simple Lie algebra. It is well-known that every real simple Lie algebra is either absolutely simple or the realification $\mathfrak{s}^{\mathbb{R}}$ of a simple complex Lie algebra $\mathfrak{s}$. In the latter case, the Cartan decomposition reads as $\mathfrak{s}^{\mathbb{R}} = \mathfrak{k} +\mathtt{i}\mathfrak{k}$ for some compact real form $\mathfrak{k}$ of $\mathfrak{s}$. In this case, it holds
$$
\mathscr{B}_{\mathfrak{s}^{\mathbb{R}}} = 2\,\mathrm{Re}(\mathscr{B}_{\mathfrak{s}}) \,\, , \quad
\mathscr{B}_{\mathfrak{k}} = \mathrm{Re}(\mathscr{B}_{\mathfrak{s}})|_{\mathfrak{k} \otimes \mathfrak{k}} \,\, .
$$
The \emph{canonical Hermitian structure} of a complex semisimple Lie group $\mathsf{S}$, with Lie algebra $\mathfrak{s}$, is the pair $(J,g)$, where the complex structure $J$ is given by 
$$
JX \coloneqq \mathtt{i}X \quad \text{for all $X \in \mathfrak{s}^{\mathbb{R}}$}
$$
and $g$ is the canonical metric \eqref{eq:canmet}. Notice that the canonical Hermitian structure does not depend on the choice of the compact real form $\mathfrak{k} \subset \mathfrak{s}$, in the sense that choosing a different compact real form produces a Hermitian structure which is holomorphically isometric to $(J,g)$ by an automorphism of $\mathsf{S}$.

\begin{proposition} \label{prop:complexss}
The canonical Hermitian structure of a simply-connected, complex semisimple Lie group is BAS reduced.
\end{proposition}

\begin{proof}
The argument is similar to the proof of Proposition \ref{prop:cmpctBAS}. Let $(\mathsf{S},J,g)$ be a complex semisimple Lie group, with Lie algebra $\mathfrak{s}$, endowed with its canonical Hermitian structure. Without loss of generality, we assume that $\mathsf{S}$ is simple. Consider the realification $\mathfrak{s}^{\mathbb{R}} = \mathfrak{k}+\mathtt{i}\mathfrak{k}$, where $\mathfrak{k}$ is a compact real form of $\mathfrak{s}$.

To show that $(\mathsf{S},J,g)$ is BAS, we enlarge the group of holomorphic isometries and find a suitable reductive complement satisfying the naturally reductive condition. To this end, consider the real Lie algebras
$$
\mathfrak{g} \coloneqq \mathfrak{s}^{\mathbb{R}} \oplus \mathfrak{k} \,\, , \quad \mathfrak{h} \coloneqq \{(U,U) : U \in \mathfrak{k}\} \subset \mathfrak{g}
$$
and define
$$
\mathfrak{m} \coloneqq \mathfrak{m}_1 +\mathfrak{m}_2 \,\, , \quad \mathfrak{m}_1 \coloneqq \{(U,2U) : U \in \mathfrak{k}\} \,\, , \quad \mathfrak{m}_2 \coloneqq \{(\mathtt{i}U,0) : U \in \mathfrak{k}\} \,\, .
$$
Then, it is routine to check that
$$
\mathfrak{g} = \mathfrak{h} +\mathfrak{m} \,\, , \quad
\mathfrak{h} \cap \mathfrak{m}_1 = \mathfrak{h} \cap \mathfrak{m}_2 = \mathfrak{m}_1 \cap \mathfrak{m}_2 = \{0\} \,\, , \quad
[\mathfrak{h},\mathfrak{m}_1] \subset \mathfrak{m}_1 \,\, , \quad
[\mathfrak{h},\mathfrak{m}_2] \subset \mathfrak{m}_2 \,\, .
$$
Moreover, since $\mathfrak{k}$ is simple, it follows that $\mathfrak{m}_1$ and $\mathfrak{m}_2$ are irreducible ${\rm ad}(\mathfrak{h})$-modules, each of which is equivalent to the adjoint action of $\mathfrak{k}$ on itself. In particular, it follows that
\begin{equation} \label{eq:m0Gss}
\{X \in \mathfrak{m} : [\mathfrak{h},X] = 0 \} = \{0\} \,\, .
\end{equation}
We denote by $\mathsf{G}$ the simply-connected Lie group with Lie algebra $\mathfrak{g}$ and we observe that, by construction, $\mathsf{G}$ acts transitively on $(\mathsf{S},J,g)$ by holomorphic isometries. We also introduce the ${\rm ad}(\mathfrak{g})$-invariant bilinear form $Q$ on $\mathfrak{g}$ defined by
$$
Q\big((U_1+\mathtt{i}V_1,W_1),(U_2+\mathtt{i}V_2,W_2)\big) \coloneqq \mathscr{B}_{\mathfrak{s}^{\mathbb{R}}}(U_1+\mathtt{i}V_1,U_2+\mathtt{i}V_2) -\mathscr{B}_{\mathfrak{k}}(W_1,W_2) \,\, .
$$
Then, it is straightforward to check that
$$\begin{gathered}
Q(\mathfrak{h},\mathfrak{m}_1) = Q(\mathfrak{h},\mathfrak{m}_2) = Q(\mathfrak{m}_1,\mathfrak{m}_2) = 0 \,\, , \\
Q((U_1,2U_1),(U_2,2U_2)) = - \mathscr{B}_{\mathfrak{s}^{\mathbb{R}}}(U_1,U_2) \,\, , \quad
Q((\mathtt{i}U_1,0),(\mathtt{i}U_2,0)) = \mathscr{B}_{\mathfrak{s}^{\mathbb{R}}}(\mathtt{i}U_1,\mathtt{i}U_2)
\end{gathered}$$
for all $U_1, U_2 \in \mathfrak{k}$. Moreover
$$
Q((U_1,U_1),(U_2,U_2)) = \tfrac{1}{2}\mathscr{B}_{\mathfrak{s}^{\mathbb{R}}}(U_1,U_2) \,\, ,
$$
and so $Q|_{\mathfrak{h} \otimes \mathfrak{h}}$ is non-degenerate. This shows that $(\mathsf{S},g)$ is $\mathsf{G}$-naturally reductive, and so $(\mathsf{S},J,g)$ is BAS by Theorem \ref{thm:BASnatrad}. Moreover, $(\mathsf{S},J,g)$ is reduced by \eqref{eq:m0Gss}.
\end{proof}

\begin{remark}
By \cite[Theorem C]{MR0146301}, every simply-connected, complex semisimple Lie group $\mathsf{S}$ has a torsion-free cocompact lattice $\Lambda \subset \mathsf{S}$. Therefore, each quotient $\Lambda \backslash\mathsf{S}$ provides an example of a compact manifold with infinite fundamental group that admits a balanced BAS structure. Moreover, if $\mathfrak{k}$ is a compact real form of the Lie algebra $\mathfrak{s}$ of $\mathsf{S}$, and $\mathsf{K}$ denotes the corresponding connected subgroup, then the canonical Hermitian structure of $\mathsf{S}$ induces a Riemannian metric on the quotient $\mathsf{S}/\mathsf{K}$ that turns the principal $\mathsf{K}$-bundle $\pi: \mathsf{S} \to \mathsf{S}/\mathsf{K}$ into a Riemannian submersion over a non-compact Riemannian symmetric space, as predicted by \cite[Theorem 5.5]{MR0787113}. However, we emphasize that the base $\mathsf{S}/\mathsf{K}$ is never Hermitian, and the projection $\pi$ is not holomorphic.
\end{remark}

We now prove Theorem \ref{thm:noncompact-structure-3} and Proposition \ref{prop:noncmpctBAS}. Together with Proposition \ref{prop:complexss}, these results imply Theorem \ref{thm:MAIN-noncpt}.

\begin{theorem} \label{thm:noncompact-structure-3}
Let $(M,J,g)$ be a complete, simply-connected BAS manifold and assume that there exists a semisimple group $\mathsf{G}_{\mathrm{nc}}$ of non-compact type acting transitively on $M$ by holomorphic isometries. Then, $(M,J,g)$ splits as the Hermitian product of complex simple Lie groups, endowed with a multiple of their canonical metric, and another BAS manifold $(M',J',g')$. Moreover, the fibres of the canonical reduction of $(M',J',g')$ are contractible, and its base is a non-compact Hermitian symmetric space.
\end{theorem}

\begin{proof}
Passing to the universal covering group, we may assume that $\mathsf{G}_{\mathrm{nc}}$ is simply-connected. We may also assume that $\mathsf{G}_{\mathrm{nc}}$ acts almost-effectively on $M$. Since $M$ is simply-connected, the isotropy subgroup $\mathsf{H}$ of $\mathsf{G}_{\mathrm{nc}}$ at $p$ is connected. Let
$$
\mathfrak{g}_{\mathrm{nc}} = \mathfrak{g}_1 \oplus {\dots} \oplus \mathfrak{g}_r \oplus \tilde{\mathfrak{g}}_1 \oplus {\dots} \oplus \tilde{\mathfrak{g}}_s
$$
be the decomposition of the Lie algebra $\mathfrak{g}_{\mathrm{nc}}$ of $\mathsf{G}_{\mathrm{nc}}$ into simple ideals, with
$\mathfrak{g}_i$ absolutely simple of non-compact type and $\tilde{\mathfrak{g}}_j = (\mathfrak{s}_j)^{\mathbb{R}}$ realification of complex simple Lie algebras. Let us consider a choice of Cartan decompositions
$$
\mathfrak{g}_i = \mathfrak{k}_i + \mathfrak{p}_i \,\, , \quad \tilde{\mathfrak{g}}_j = \tilde{\mathfrak{k}}_j + \mathtt{i}\tilde{\mathfrak{k}}_j
$$
such that the Lie algebra $\mathfrak{h}$ of $\mathsf{H}$ satisfies
$$
\mathfrak{h} \subset \mathfrak{k} \oplus \tilde{\mathfrak{k}} \,\, , \quad
\text{with} \quad
\mathfrak{k} \coloneqq \mathfrak{k}_1 \oplus {\dots} \oplus \mathfrak{k}_r \quad
\text{and} \quad
\tilde{\mathfrak{k}} \coloneqq \tilde{\mathfrak{k}}_1 \oplus {\dots} \oplus \tilde{\mathfrak{k}}_s \,\, .
$$
We also denote by $\mathfrak{v}$ the orthogonal complement of $\mathfrak{h}$ in $\mathfrak{g}_{\mathrm{nc}}$ with respect to $\mathscr{B}_{\mathfrak{g}_{\mathrm{nc}}}$. We divide the proof into three steps.
\smallskip

\noindent \emph{Step 1: The manifold $(M,J,g)$ splits as a Hermitian product.}

By \cite[Theorem 5.2]{MR0787113}\footnote{Notice that Claude and chatGPT found an gap in the proof of \cite[Theorem 5.2]{MR0787113}. However, claim (i), claim (ii), and the last claim are true. Indeed, the proof can be corrected following the same line of the proof of Proposition \ref{prop:noncompact-structure1}.} 
and Theorem \ref{thm:BASnatrad}, it follows that any isometry of $(M,J,g)$ isotopic to the identity is holomorphic. Therefore, by \cite[Theorem 5.2]{MR0787113}, the Hermitian structure $(J,g)$ is $\mathrm{ad}(\mathfrak{k} \oplus \tilde{\mathfrak{k}})$-invariant.

By \cite[Theorem 5.2]{MR0787113}, it follows that $\mathfrak{h}$ is an ideal of $\mathfrak{k} \oplus \tilde{\mathfrak{k}}$. We also observe that, for every $1 \leq j \leq s$, the intersection $\tilde{\mathfrak{h}}_j \coloneqq \mathfrak{h} \cap \tilde{\mathfrak{k}}_j$ is an ideal of the simple Lie algebra $\tilde{\mathfrak{k}}_j$. Hence, either $\tilde{\mathfrak{h}}_j = \{0\}$ or $\tilde{\mathfrak{h}}_j = \tilde{\mathfrak{k}}_j$. We claim that $\tilde{\mathfrak{h}}_j = \{0\}$ for all $1 \leq j \leq s$. Indeed, assume by contradiction that, up to reordering, $\tilde{\mathfrak{h}}_1= \tilde{\mathfrak{k}}_1$. Then, $\mathtt{i}\tilde{\mathfrak{k}}_1 \subset \mathfrak{v}$ is an irreducible $\mathrm{ad}(\tilde{\mathfrak{k}})$-module that is inequivalent to any other irreducible $\mathrm{ad}(\tilde{\mathfrak{k}})$-submodule of $\mathfrak{v}$. Therefore, it follows by Schur's Lemma that $J(\mathtt{i}\tilde{\mathfrak{k}}_1) = \mathtt{i}\tilde{\mathfrak{k}}_1$, which contradicts the fact that $\tilde{\mathfrak{k}}_1$ is a compact simple Lie algebra. Since $\mathfrak{h}$ is an ideal of $\mathfrak{k} \oplus \tilde{\mathfrak{k}}$, we have $[\tilde{\mathfrak{k}}_j,\mathfrak{h}] \subset \tilde{\mathfrak{k}}_j \cap \mathfrak{h} = \{0\}$. Thus, the projection of $\mathfrak{h}$ onto each $\tilde{\mathfrak{k}}_j$ is central, hence trivial. Consequently, $\mathfrak{h} \subset \mathfrak{k}$.

For the sake of notation, we set
$$
\mathfrak{g} \coloneqq \mathfrak{g}_1 \oplus {\dots} \oplus \mathfrak{g}_r \,\, , \quad
\mathfrak{p} \coloneqq \mathfrak{p}_1 \oplus {\dots} \oplus \mathfrak{p}_r \,\, ,
$$
and we denote by $\mathfrak{a}$ the orthogonal complement of $\mathfrak{h}$ in $\mathfrak{k}$ with respect to $\mathscr{B}_{\mathfrak{g}}$. Since the adjoint representation of a compact simple Lie algebra is irreducible and non-trivial, the complement
$$
\mathfrak{q} \coloneqq \mathfrak{a} + \mathfrak{p}
$$
contains no irreducible $\mathrm{ad}(\tilde{\mathfrak{k}})$-submodule equivalent to the $\mathrm{ad}(\tilde{\mathfrak{k}})$-module $\tilde{\mathfrak{k}}_j$, for every $1\leq j\leq s$. Moreover, for the same reason, for every $1 \leq j_1 < j_2 \leq s$, the $\mathrm{ad}(\tilde{\mathfrak{k}})$-modules $\tilde{\mathfrak{k}}_{j_1}$ and $\tilde{\mathfrak{k}}_{j_2}$ are not equivalent. Therefore, the Hermitian structure preserves the subspaces
$$
\mathfrak{q} \, , \,\, \tilde{\mathfrak{g}}_1 \, , \,\, {\dots} \, , \,\, \tilde{\mathfrak{g}}_s \,\, 
$$
from which we obtain the splitting of $(M,J,g)$ as the Hermitian product
$$
(M,J,g) = 
(M' , J', g') \times (\tilde{\mathsf{G}}_1,J_1,g_1) \times {\dots} \times (\tilde{\mathsf{G}}_s,J_s,g_s) \,\, .
$$
Here, $M' = \mathsf{G}/\mathsf{H}$, where $\mathsf{G}$ is the Lie subgroup of $\mathsf{G}_{\mathrm{nc}}$ with Lie algebra $\mathfrak{g}$, $\tilde{\mathsf{G}}_j$ is the connected Lie subgroup of $\mathsf{G}_{\mathrm{nc}}$ with Lie algebra $\tilde{\mathfrak{g}}_j$, $(J',g')$ is the restriction of $(J,g)$ to $\mathfrak{q}$, and $(J_j,g_j)$ is the restriction of $(J,g)$ to $\tilde{\mathfrak{g}}_j$.

\smallskip

\noindent \emph{Step 2: For every $1 \leq j \leq s$, the pair $(J_j,g_j)$ is the canonical Hermitian structure on $\tilde{\mathsf{G}}_j$, up to scaling.}

We prove the claim for $j=1$. Recall that $g_1$ is ${\rm ad}(\tilde{\mathfrak{k}}_1)$-invariant and that $g(\tilde{\mathfrak{k}}_1,\mathtt{i}\tilde{\mathfrak{k}}_1) = 0$. Hence
$$
g_1 = \alpha (-\mathscr{B}_{\tilde{\mathfrak{g}}_1}|_{\tilde{\mathfrak{k}}_1 \otimes \tilde{\mathfrak{k}}_1}) \oplus \beta (\mathscr{B}_{\tilde{\mathfrak{g}}_1}|_{\mathtt{i}\tilde{\mathfrak{k}}_1 \otimes \mathtt{i}\tilde{\mathfrak{k}}_1}) \,\, ,
$$
for some $\alpha, \beta >0$. Since $J_1$ is ${\rm ad}(\tilde{\mathfrak{k}}_1)$-invariant and $\tilde{\mathfrak{k}}_1$ is an irreducible ${\rm ad}(\tilde{\mathfrak{k}}_1)$-module, it follows that $J_1\tilde{\mathfrak{k}}_1$ is an irreducible ${\rm ad}(\tilde{\mathfrak{k}}_1)$-module and so either $\tilde{\mathfrak{k}}_1 \cap J_1\tilde{\mathfrak{k}}_1 = \{0\}$ or $J_1\tilde{\mathfrak{k}}_1 = \tilde{\mathfrak{k}}_1$. Since $\tilde{\mathfrak{k}}_1$ is a compact simple Lie algebra, it follows that $\tilde{\mathfrak{k}}_1 \cap J_1\tilde{\mathfrak{k}}_1 = \{0\}$, and so $J\tilde{\mathfrak{k}}_1 = \mathtt{i}\tilde{\mathfrak{k}}_1$ for dimensional reasons and by the fact that $\tilde{\mathfrak{k}}_1, \mathtt{i}\tilde{\mathfrak{k}}_1$ are $g_1$-orthogonal. Since the adjoint representation of $\tilde{\mathfrak{k}}_1$ is of real type and $J_1^2=-{\rm Id}_{\tilde{\mathfrak{g}}_1}$, it follows that
$$
J_1X=\lambda\mathtt{i}X \,\, , \quad J_1(\mathtt{i}X)=-\tfrac1\lambda X \quad \text{for all $X \in \tilde{\mathfrak{k}}_1$} \,\, ,
$$
for some $\lambda \in \mathbb{R} \setminus \{0\}$. Moreover, the Nijenhuis tensor reads as
$$
N_J(X,Y) = (1-\lambda^2)[X,Y] \quad \text{for all $X, Y \in \tilde{\mathfrak{k}}_1$}
$$
and so $\lambda = \pm 1$, which in turn implies that $\alpha = \beta$.

\smallskip

\noindent \emph{Step 3: The fibres of the canonical reduction of $(M',J',g')$ are contractible and the base is a non-compact Hermitian symmetric space.}

Recall that $\mathfrak{q} = \mathfrak{a} +\mathfrak{p}$ is a reductive complement for $\mathfrak{h}$ in $\mathfrak{g}$. By \cite[Theorem 5.2]{MR0787113}, it follows that $g(\mathfrak{a},\mathfrak{p})=0$. By \cite[Theorem 5.2]{MR0787113}, it follows that the infinitesimal automorphism algebra of $(M',J',g')$ satisfies
$$
\mathfrak{aut}(M') = \mathfrak{g} \oplus \mathfrak{a} \,\, ,
$$
while its isotropy subalgebra is
$$
\mathfrak{aut}_0(M') = (\mathfrak{h} \oplus \{0\}) + \Delta\mathfrak{a} \,\, , \quad \text{with} \quad \Delta\mathfrak{a} \coloneqq \{(U, U) : U \in \mathfrak{a}\} \,\, .
$$
Since every ideal $\mathfrak{g}_i$ is absolutely simple, it follows that $\mathfrak{a}$ does not contain any irreducible $\mathrm{ad}(\mathfrak{k})$-module that is $\mathrm{ad}(\mathfrak{k})$-equivalent to $\mathfrak{p}_i$ (see, \emph{e.g.}, \cite[7.104]{MR2371700}). Since the Hermitian structure $(J',g')$ is $\mathrm{ad}(\mathfrak{k})$-invariant, this implies that $J'\mathfrak{a} = \mathfrak{a}$ and $J'\mathfrak{p}_i = \mathfrak{p}_i$ for all $1 \leq i \leq r$. In particular, it follows that $\mathfrak{a}$ is abelian. Moreover, since $\mathfrak{h}$ is an ideal of $\mathfrak{k}$, its orthogonal complement $\mathfrak{a}$ is also an ideal of $\mathfrak{k}$. Consequently, $\mathfrak{a} \subset \mathfrak{z}(\mathfrak{k})$.

We consider now the canonical reductive decomposition
$$
\mathfrak{aut}(M') = \mathfrak{aut}_0(M') + \mathfrak{m}'
$$
as in \eqref{eq:canreddec}. By the proof of \cite[Theorem 5.2]{MR0787113}, we have $[\mathfrak{m}',\mathfrak{m}'] + \mathfrak{m}' = \mathfrak{aut}(M')$, and so $\mathfrak{aut}(M')$ coincides with the Lie algebra $\mathfrak{l}$ defined in \eqref{eq:def-lu}. We may consider then the bilinear form $Q$ on $\mathfrak{aut}(M')$ as in Proposition \ref{prop:Kostant}. By bi-invariance and the fact that $\mathfrak{g}$ is semisimple, it follows that $Q(\mathfrak{g} \oplus \{0\}, \{0\} \oplus \mathfrak{a}) = 0$. Moreover, the restriction $Q' \coloneqq Q|_{(\mathfrak{g} \oplus \{0\}) \otimes (\mathfrak{g} \oplus \{0\})}$ is still bi-invariant and non-degenerate. Since every $\mathfrak{g}_i$ is absolutely simple, there exist $\gamma_1, {\dots}, \gamma_r \in \mathbb{R} \setminus \{0\}$ such that
$$
Q' = \sum_{i=1}^r \gamma_i \, \mathscr{B}_{\mathfrak{g}_i} \,\, .
$$
This implies that $Q'(\mathfrak{k}, \mathfrak{p}) = 0$. It then follows that
$$
Q(\mathfrak{aut}_0(M'), \mathfrak{p} \oplus \{0\}) = 0 \,\, ,
$$
which in turn implies that $\mathfrak{p} \oplus \{0\} \subset \mathfrak{m}'$. Since $Q|_{\mathfrak{m}' \otimes \mathfrak{m}'}$ is positive-definite, we have $\gamma_i>0$ for every $1\leq i\leq r$. Let $\mathfrak{f}$ be the $Q$-orthogonal complement of $\mathfrak{p}\oplus\{0\}$ in $\mathfrak{m}'$. Since $Q(\{0\} \oplus \mathfrak{a}, \mathfrak{p} \oplus\{0\})=0$ and $Q(\mathfrak{aut}_0(M'),\mathfrak{m}')=0$, the projection of $\mathfrak{z}(M')= \{0\}\oplus\mathfrak{a}$ onto
$\mathfrak{m}'$ is contained in $\mathfrak{f}$. This projection is injective, since $\mathfrak{z}(M') \cap \mathfrak{aut}_0(M')=\{0\}$. Moreover,
$$
\dim(\mathfrak{f}) = \dim(\mathfrak{m}') -\dim(\mathfrak{p}) = \dim(\mathfrak{a}) \,\, ,
$$
and so $\mathfrak{f}$ coincides with the projection of $\mathfrak{z}(M')$ onto $\mathfrak{m}'$, as in the proof of Theorem \ref{thm:generalstructure}. This shows that the fibres of the canonical reduction coincide with those of the projection $\mathsf{G}/\mathsf{H} \to \mathsf{G}/\mathsf{K}$, where $\mathsf{K}$ is the connected Lie subgroup of $\mathsf{G}$ with Lie algebra $\mathfrak{k}$. In particular, they are diffeomorphic to $\mathsf{K}/\mathsf{H}$, which is a connected abelian Lie group. Since $\mathsf{G}/\mathsf{K}$ is contractible and $M'$ is simply-connected, it follows that $\mathsf{K}/\mathsf{H}$ is simply-connected, hence contractible.
\end{proof}

We observe that, in the special case when the Hermitian manifold is $\mathsf{G}$-naturally reductive with respect to a semisimple Lie group $\mathsf{G}$ of non-compact type, the following strengthened result holds.

\begin{proposition} \label{prop:noncompact-structure1}
Let $(M,J,g)$ be a simply-connected Hermitian manifold, and let $\mathsf{G}$ be a semisimple Lie group of non-compact type acting transitively on $M$ by holomorphic isometries. If $(M,g)$ is $\mathsf{G}$-naturally reductive, then $(M,J,g)$ is a Hermitian symmetric space of non-compact type.
\end{proposition}

\begin{proof}
By \cite[Theorem 5.2]{MR0787113}, we may assume that $\mathsf{G} = \mathrm{Aut}^{\mathrm{o}}(M)$. Let $\mathsf{H}$ be the isotropy of $\mathsf{G}$ at a point $p \in M$ and let $\mathfrak{h}$ denote its Lie algebra. Consider the canonical reductive decomposition
$$
\mathfrak{g} = \mathfrak{h} + \mathfrak{m}
$$
as in Theorem \ref{thm:BASnatrad}. By the proof of \cite[Theorem 5.2]{MR0787113}, we have $\mathfrak{g} = [\mathfrak{m},\mathfrak{m}] +\mathfrak{m}$ and so we may consider the bilinear form $Q$ on $\mathfrak{g}$ as in Proposition \ref{prop:Kostant}. Since $Q(\mathfrak{h},\mathfrak{m}) = 0$ and $Q|_{\mathfrak{m} \otimes \mathfrak{m}} >0$, it follows that
\begin{equation} \label{eq:ind1}
\mathrm{ind}_{-}(Q) \leq \dim(\mathfrak{h}) \,\, ,
\end{equation}
where $\mathrm{ind}_{-}(Q)$ denotes the number of negative eigenvalues of $Q$. Let
$$
\mathfrak{g} = \mathfrak{g}_1 \oplus {\dots} \oplus \mathfrak{g}_r \oplus \tilde{\mathfrak{g}}_1 \oplus {\dots} \oplus \tilde{\mathfrak{g}}_s
$$
be the decomposition of the Lie algebra $\mathfrak{g}$ of $\mathsf{G}$ into simple ideals, with
$\mathfrak{g}_i$ absolutely simple of non-compact type and $\tilde{\mathfrak{g}}_j = (\mathfrak{s}_j)^{\mathbb{R}}$ the realifications of complex simple Lie algebras. Let us consider a choice of Cartan decompositions
$$
\mathfrak{g}_i = \mathfrak{k}_i + \mathfrak{p}_i \,\, , \quad \tilde{\mathfrak{g}}_j = \tilde{\mathfrak{k}}_j + \mathtt{i}\tilde{\mathfrak{k}}_j
$$
such that
$$
\mathfrak{h} \subset \mathfrak{k} \coloneqq \mathfrak{k}_1 \oplus {\dots} \oplus \mathfrak{k}_r \oplus \tilde{\mathfrak{k}}_1 \oplus {\dots} \oplus \tilde{\mathfrak{k}}_s \,\, .
$$
Since $Q$ is bi-invariant and non-degenerate,
there exist $\gamma_i \in \mathbb{R}\setminus\{0\}$ and
$\alpha_j, \beta_j \in \mathbb{R}$, with $(\alpha_j,\beta_j)\neq(0,0)$, such that
$$
Q = \sum_{i=1}^r \gamma_i\,\mathscr{B}_{\mathfrak{g}_i}
+ \sum_{j=1}^s \big(
\alpha_j\,\operatorname{Re}(\mathscr{B}_{\mathfrak{s}_j})
+ \beta_j\,\operatorname{Im}(\mathscr{B}_{\mathfrak{s}_j})
\big) \,\, .
$$
Fix $1 \leq i \leq r$. Since $\mathfrak{p}_i \neq \{0\}$, this implies that $Q(\mathfrak{k}_i, \mathfrak{p}_i) = 0$, and so $\mathfrak{p}_i \subset \mathfrak{m}$. It then follows that $\gamma_i >0$, and so
\begin{equation} \label{eq:ind2}
\mathrm{ind}_{-}(Q|_{\mathfrak{g}_i \otimes \mathfrak{g}_i}) = \dim(\mathfrak{k}_i) \,\, .
\end{equation}
Fix $1 \leq j \leq s$ and a $(-\mathscr{B}_{\mathfrak{k}_j})$-orthonormal basis for $\mathfrak{k}_j$. With respect to such a basis, the bilinear form $Q|_{\tilde{\mathfrak{g}}_j \otimes \tilde{\mathfrak{g}}_j}$ takes the form
$$
\left(\begin{array}{cc}
-\alpha_j & -\beta_j \\
-\beta_j & \alpha_j
\end{array}\right) \otimes \mathrm{Id}_{\mathfrak{k}} \,\, ,
$$
from which it follows that
\begin{equation} \label{eq:ind3}
\mathrm{ind}_{-}(Q|_{\tilde{\mathfrak{g}}_j \otimes \tilde{\mathfrak{g}}_j}) = \dim(\mathfrak{k}_j) \,\, .
\end{equation}
From \eqref{eq:ind1}, \eqref{eq:ind2} and \eqref{eq:ind3}, we get $\dim(\mathfrak{k}) \leq \dim(\mathfrak{h})$. Hence, $\mathfrak{h} = \mathfrak{k}$ is a maximal compact subalgebra of $\mathfrak{g}$, and this concludes the proof.
\end{proof}

\begin{remark}
Recall that if $\mathsf{G}$ is a connected real semisimple Lie group whose Lie algebra $\mathfrak{g}$ contains a simple ideal that is the realification of a simple complex Lie algebra, then there exists no compact subgroup $\mathsf{K} \subset \mathsf{G}$ such that the quotient $\mathsf{G}/\mathsf{K}$ is a Hermitian symmetric space. Therefore, under the assumptions of Proposition \ref{prop:noncompact-structure1}, it follows that the Lie algebra of $\mathsf{G}$ splits as a direct sum of absolutely simple ideals.
\end{remark}

Finally, in complete analogy with Proposition \ref{prop:cmpctBAS}, the following result holds.

\begin{proposition} \label{prop:noncmpctBAS}
Let $\pi: \mathsf{G}/\mathsf{H} \to \mathsf{G}/\mathsf{H}\mathsf{A}$ be a homogeneous abelian principal bundle over a Hermitian symmetric space, where $\mathsf{G}$ is a semisimple Lie group of non-compact type, and let $\mathsf{A}$ be a connected abelian Lie group. Assume that $\mathsf{A}' \times \mathsf{G}/\mathsf{H}$ admits a left-$(\mathsf{A}' \times \mathsf{G})$ and right-$(\mathsf{A}' \times \mathsf{A})$-invariant Hermitian structure $(J,g)$ and that $g$ induces on $\mathsf{G}/\mathsf{H}\mathsf{A}$ the symmetric metric. Then, $(\mathsf{A}' \times \mathsf{G}/\mathsf{H},J,g)$ is BAS.
\end{proposition}

\begin{proof}
The proof proceeds along the same lines as the proof of Proposition \ref{prop:cmpctBAS}. We denote by $\mathfrak{g}$, $\mathfrak{h}$, $\mathfrak{a}$, $\mathfrak{a}'$ the Lie algebras of $\mathsf{G}$, $\mathsf{H}$, $\mathsf{A}$ and $\mathsf{A}'$, respectively. For the sake of notation, we set $\tilde{\mathfrak{a}} \coloneqq \mathfrak{a}' \oplus \mathfrak{a}$ and $\tilde{\mathfrak{g}} \coloneqq \mathfrak{a}' \oplus \mathfrak{g}$. Denote by $\mathfrak{g} = \mathfrak{g}_1 \oplus {\dots} \oplus \mathfrak{g}_s$ the decomposition of $\mathfrak{g}$ into simple ideals. Let $\mathfrak{g}_i = \mathfrak{k}_i + \mathfrak{b}_i$ be a Cartan decomposition of $\mathfrak{g}_i$, for all $1 \leq i \leq s$, such that $\mathfrak{h} \subset \mathfrak{k} \coloneqq \sum_i \mathfrak{k}_i$. Let also $\mathfrak{b} = \mathfrak{b}_1 + {\dots} + \mathfrak{b}_s$. Denote by $c_i>0$ the coefficients so that
$$
g|_{\mathfrak{b}_i \otimes \mathfrak{b}_i} = c_i \mathscr{B}_{\mathfrak{g}_i}|_{\mathfrak{b}_i \otimes \mathfrak{b}_i} \quad \text{for all $1 \leq i \leq s$}
$$
and set
$$
Q \coloneqq (-Q_0) + c_1\mathscr{B}_1 + {\dots} + c_s \mathscr{B}_s \,\, ,
$$
for some $Q_0 \in \mathrm{Sym}_{+}(\mathfrak{a}')$. Notice that, by Schur's Lemma, $g(\tilde{\mathfrak{a}},\mathfrak{b}) = 0$ and $J\tilde{\mathfrak{a}} = \tilde{\mathfrak{a}}$. Let us consider the endomorphism $S \in \mathrm{Sym}_+(\tilde{\mathfrak{a}},-Q)$ defined by
$$
g(U,V) = -Q(SU,V) \quad \text{for all $U,V \in \tilde{\mathfrak{a}}$} \,\, .
$$
To show that $(\mathsf{A}' \times \mathsf{G}/\mathsf{H},J,g)$ is BAS, we enlarge the group of holomorphic isometries and find a suitable reductive complement satisfying the naturally reductive condition. To this end, consider the real Lie algebras
$$
\hat{\mathfrak{g}} \coloneqq \tilde{\mathfrak{g}} \oplus \tilde{\mathfrak{a}} \,\, , \quad \hat{\mathfrak{h}} \coloneqq \big\{(H+U,U) : H \in \mathfrak{h} \,\, , U \in \tilde{\mathfrak{a}}\big\}
$$
and
$$
\mathfrak{m} \coloneqq \mathfrak{m}_1 + \mathfrak{m}_2 \,\, , \quad
\mathfrak{m}_1 \coloneqq \big\{(SU,(S+\mathrm{Id})U) : U \in \tilde{\mathfrak{a}} \} \,\, , \quad
\mathfrak{m}_2 \coloneqq \{(X,0) : X \in \mathfrak{b}\} \,\, .
$$
Then, it is routine to check that
$$\begin{gathered}
\hat{\mathfrak{g}} = \hat{\mathfrak{h}} +\mathfrak{m} \,\, , \quad
\hat{\mathfrak{h}} \cap \mathfrak{m}_1 = \hat{\mathfrak{h}} \cap \mathfrak{m}_2 = \mathfrak{m}_1 \cap \mathfrak{m}_2 = \{0\} \,\, , \quad
[\hat{\mathfrak{h}},\mathfrak{m}_1] = \{0\} \,\, , \quad
[\hat{\mathfrak{h}},\mathfrak{m}_2] \subset \mathfrak{m}_2 \,\, .
\end{gathered}$$
We denote by $\widehat{\mathsf{G}}$ the simply-connected Lie group with Lie algebra $\hat{\mathfrak{g}}$, and we observe that, by construction, $\widehat{\mathsf{G}}$ acts transitively on $(\mathsf{A}' \times\mathsf{G}/\mathsf{H},J,g)$ by holomorphic isometries. The evaluation map at the origin of $\mathsf{A}' \times \mathsf{G}/\mathsf{H}$ induces a linear isomorphism $\mathrm{ev}: \mathfrak{m} \to \tilde{\mathfrak{a}} +\mathfrak{b}$ such that
$$
\mathrm{ev}|_{\mathfrak{m}_1} : (SU,(S+\mathrm{Id})U) \mapsto -U \,\, , \quad
\mathrm{ev}|_{\mathfrak{m}_2} : (X,0) \mapsto X \,\, .
$$
We also introduce the $\mathrm{ad}(\hat{\mathfrak{g}})$-invariant bilinear form $\widehat{Q}$ on $\hat{\mathfrak{g}}$ defined by
$$
\widehat{Q}\big((E_1,U_1),(E_2,U_2)\big) \coloneqq Q(E_1,E_2) -Q(S(S+\mathrm{Id})^{-1}U_1,U_2) \,\, .
$$
Then, it is straightforward to check that
$$\begin{gathered}
\widehat{Q}(\hat{\mathfrak{h}},\mathfrak{m}_1) = \widehat{Q}(\hat{\mathfrak{h}},\mathfrak{m}_2) = \widehat{Q}(\mathfrak{m}_1,\mathfrak{m}_2) = 0 \,\, , \\
\widehat{Q}((X_1,0),(X_2,0)) = g(X_1,X_2) \,\, , \quad
\widehat{Q}((SU_1,(S+\mathrm{Id})U_1),(SU_2,(S+\mathrm{Id})U_2)) = g(U_1,U_2)
\end{gathered}$$
for all $X_1, X_2 \in \mathfrak{b}$ and $U_1,U_2 \in \tilde{\mathfrak{a}}$. Moreover
$$
\widehat{Q}((H_1+U_1,U_1),(H_2+U_2,U_2)) = Q(H_1,H_2)+Q\big((S+\mathrm{Id})^{-1}U_1,U_2\big)
$$
for all $H_1, H_2 \in \mathfrak{h}$, and so $\widehat{Q}|_{\hat{\mathfrak{h}} \otimes \hat{\mathfrak{h}}}$ is non-degenerate. This proves that $(\mathsf{A}' \times\mathsf{G}/\mathsf{H},g)$ is $\widehat{\mathsf{G}}$-naturally reductive, and so $(\mathsf{A}' \times\mathsf{G}/\mathsf{H},J,g)$ is BAS by Theorem \ref{thm:BASnatrad}.
\end{proof}

\medskip
\section{Bismut--Ambrose--Singer nilmanifolds} \label{sect:nil} \setcounter{equation} 0

The aim of this section is to prove Theorem \ref{thm:MAIN-nil}, namely the classification of simply-connected nilpotent Lie groups admitting a left-invariant BAS structure. To this end, we consider the following special class of real nilpotent Lie algebras.

\begin{definition} \label{def:k-nil}
Let $(\mathfrak{k},\langle\cdot,\cdot\rangle_{\mathfrak{k}})$ be a compact Lie algebra endowed with a bi-invariant Euclidean scalar product. Let also $\varphi: \mathfrak{k} \to \mathfrak{u}(V,I,\langle\cdot,\cdot\rangle_{V})$ be a finite-dimensional, unitary representation without trivial submodules. We define a \emph{$\mathfrak{k}$-nilpotent triple} to be the data $(\mathfrak{n},I,g)$, where:
\begin{itemize}
\item[$i)$] $\mathfrak{n}$ is the real vector space given by the direct sum $\mathfrak{n} \coloneqq \mathfrak{k} + V$ endowed with the Lie bracket $[\cdot,\cdot]_{\mathfrak{n}}$ whose only non-zero components are given by
\begin{equation} \label{eq:knilbrack}
\langle [v_1,v_2]_{\mathfrak{n}}, K \rangle_{\mathfrak{k}} \coloneqq \langle \varphi(K)v_1,v_2 \rangle_{V}
\quad \text{for all } v_1,v_2 \in V,\ K \in \mathfrak{k};
\end{equation}
\item[$ii)$] $I$ is the transverse complex structure of $\mathfrak{n}$ defined on $V$;
\item[$iii)$] $g$ is the scalar product defined by $g \coloneqq \langle\cdot,\cdot\rangle_{\mathfrak{k}} \oplus \langle\cdot,\cdot\rangle_{V}$.
\end{itemize}
A Lie algebra $\mathfrak{n}$ is \emph{$\mathfrak{k}$-nilpotent} if it is the underlying Lie algebra of a $\mathfrak{k}$-nilpotent triple $(\mathfrak{n},I,g)$.
\end{definition}

We remark that a $\mathfrak{k}$-nilpotent Lie algebra $\mathfrak{n}$ is $2$-step nilpotent, and its centre $\mathfrak{z}(\mathfrak{n})$ coincides with the underlying vector space of $\mathfrak{k}$. Moreover, a direct computation shows that, for every $\mathfrak{k}$-nilpotent triple $(\mathfrak{n},I,g)$, any central left-invariant vector field preserves both the metric $g$ and the transverse complex structure $I$. Finally, it follows by construction that
\begin{equation} \label{eq:abelianI}
[Iv_1,Iv_2]_{\mathfrak{n}} = [v_1,v_2]_{\mathfrak{n}} \quad \text{for all $v_1,v_2 \in V$} \,\, .
\end{equation}

As we shall see below, Heisenberg Lie algebras provide examples of $\mathfrak{k}$-nilpotent Lie algebras, both for abelian and semisimple $\mathfrak{k}$, as well as examples of $2$-step nilpotent Lie algebras which are not $\mathfrak{k}$-nilpotent.

\begin{example}
The \emph{real Heisenberg algebra} $\mathfrak{h}_{2n+1}(\mathbb{R})$ is $\mathfrak{k}$-nilpotent, with $\mathfrak{k} = \mathbb{R}$ and $\varphi: \mathbb{R} \to \mathfrak{u}(n)$ the direct sum of $n$ copies of the standard rotational representation of $\mathbb{R}$ on $\mathbb{C}$.

On the other hand, the realification of the \emph{complex Heisenberg algebra} $\mathfrak{h}_{2n+1}(\mathbb{C})$ is a $2$-step nilpotent Lie algebra which is not $\mathfrak{k}$-nilpotent. Indeed, we have
$$
\mathfrak{h}_{2n+1}(\mathbb{C})^{\mathbb{R}} = \mathrm{span}\big\{z_1,z_2\big\} + \mathrm{span}\big\{e^i_1,e^i_2,e^i_3,e^i_4\big\}_{1 \leq i \leq n} \,\, ,
$$
with only non-zero Lie brackets given by
\begin{equation} \label{eq:defhC}
\big[e^i_1,e^i_3\big] = -\big[e^i_2,e^i_4\big] = z_1 \,\, , \quad \big[e^i_1,e^i_4\big] = \big[e^i_2,e^i_3\big] = z_2 \quad \text{for $1 \leq i \leq n$} \,\, .
\end{equation}
Note that the Lie brackets \eqref{eq:defhC} do not fit into the classification result given in Proposition \ref{prop:BASnilalg}, and so $\mathfrak{h}_{2n+1}(\mathbb{C})^{\mathbb{R}}$ is not $\mathbb{R}^2$-nilpotent. Since every compact $2$-dimensional real Lie algebra $\mathfrak{k}$ is abelian, the claim follows.

Finally, we recall that the \emph{quaternionic Heisenberg algebra} $\mathfrak{h}_{4n+3}(\mathbb{H})$ is the real Lie algebra
$$
\mathfrak{h}_{4n+3}(\mathbb{H}) = \mathrm{span}\big\{z_1,z_2,z_3\big\} + \mathrm{span}\big\{e^i_1,e^i_2,e^i_3,e^i_4\big\}_{1 \leq i \leq n} \,\, ,
$$
with only non-zero Lie brackets given by
$$
\big[e^i_1,e^i_2\big] = -\big[e^i_3,e^i_4\big] = z_1 \,\, , \quad \big[e^i_1,e^i_3\big] = \big[e^i_2,e^i_4\big] = z_2 \,\, , \quad \big[e^i_1,e^i_4\big] = -\big[e^i_2,e^i_3\big] = z_3 \quad \text{for $1 \leq i \leq n$} \,\, .
$$
We notice that $\mathfrak{h}_{4n+3}(\mathbb{H})$ is $\mathfrak{k}$-nilpotent, with $\mathfrak{k}=\mathfrak{su}(2)$ and $\varphi: \mathfrak{su}(2) \to \mathfrak{u}(2n)$ the direct sum of $n$ copies of the standard representation of $\mathfrak{su}(2)$ on $\mathbb{C}^{2}$.
\end{example}

\begin{example} \label{ex:n11}
We define the Lie algebra
$$
\mathfrak{n}^8_{1,1} \coloneq \mathrm{span}\big\{z_1,z_2\big\} + \mathrm{span}\big\{e_1,e_2,e_3,e_4,e_5,e_6\big\} \,\, ,
$$
with only non-zero Lie brackets given by
$$
[e_1,e_2] = z_1 \,\, , \quad
[e_3,e_4] = z_2 \,\, , \quad
[e_5,e_6] = z_1 +z_2 \,\, .
$$
Then, we notice that $\mathfrak{n}^8_{1,1}$ is a $\mathfrak{k}$-nilpotent Lie algebra, with $\mathfrak{k}= \mathbb{R}^2$ and
$$
\varphi: \mathbb{R}^2 \to \mathfrak{u}(3) \,\, , \quad
\varphi(t,s) \coloneqq 
\left( \begin{array}{ccc}
\mathtt{i}t & 0 & 0 \\
0 & \mathtt{i}s & 0 \\
0 & 0 & \mathtt{i}(t+s) \\
\end{array} \right) \,\, .
$$
\end{example}

We remark that nilpotent Lie algebras constructed via orthogonal representations of compact Lie groups have been studied in \cite{MR1702605}. The importance of Definition \ref{def:k-nil} lies in the following characterisation of simply-connected nilpotent Lie groups admitting a left-invariant BAS structure.

\begin{theorem} \label{thm:nilBAS}
Let $(\mathsf{N}, J, g)$ be a simply-connected nilpotent Lie group, with Lie algebra $\mathfrak{n}$, endowed with a left-invariant Hermitian structure. Then, $(\mathsf{N}, J, g)$ is BAS if and only if $J$ preserves the centre $\mathfrak{z}(\mathfrak{n})$ of $\mathfrak{n}$ and $(\mathfrak{n},J|_{\mathfrak{z}(\mathfrak{n})^{\perp}},g)$ is a $\mathfrak{k}$-nilpotent triple, as in Definition \ref{def:k-nil}, with $\mathfrak{k}$ abelian.
\end{theorem}

\begin{proof}
Assume that $(\mathsf{N}, J, g)$ is BAS. By \cite[Theorem 4.3]{MR0787113}, $\mathfrak{n}$ is at most $2$-step nilpotent. By \cite[Proposition 3.3]{MR2533671}, the complex structure $J$ is \emph{nilpotent} in the sense of \cite{MR1665327}. By \cite[Proposition 1.5]{ZZpt}, $J$ is \emph{abelian}, that is, $[JX,JY]=[X,Y]$ for all $X, Y \in \mathfrak{n}$. In particular, the centre $\mathfrak{z}(\mathfrak{n})$ of $\mathfrak{n}$ and its $g$-orthogonal complement $\mathfrak{z}(\mathfrak{n})^{\perp}$ are $J$-invariant.

Let $\mathsf{G}$ be the identity component of the group of holomorphic isometries of $(\mathsf{N}, J, g)$, and let $\mathsf{H}$ be the stabilizer of the identity element $e \in \mathsf{N}$. Denote by $\mathfrak{g}$ and $\mathfrak{h}$ the Lie algebras of $\mathsf{G}$ and $\mathsf{H}$, respectively. By a straightforward application of \cite[Theorem 2]{MR0661539}, it follows that
\begin{equation} \label{eq:Gnilbas}
\mathsf{G} = \mathsf{N} \rtimes \mathsf{H} \,\, ,
\end{equation}
where $\mathsf{N}$ acts on itself by left translations, and $\mathsf{H}$ is the connected subgroup with Lie algebra $\mathfrak{h} = \mathfrak{u}(\mathfrak{n},J,g) \cap \mathrm{Der}(\mathfrak{n})$, acting on $\mathsf{N}$ through the identification given by the exponential map $\exp: \mathfrak{n} \to \mathsf{N}$. By Theorem \ref{thm:BASnatrad}, there exists a linear map $\rho: \mathfrak{n} \to \mathfrak{h}$ such that
\begin{equation} \label{eq:nilbascompl}
\mathfrak{m} \coloneqq \{(X,\rho(X)) : X \in \mathfrak{n}\}
\end{equation}
is a reductive complement at $e \in \mathsf{N}$ for the $\mathsf{G}$-action and $(\mathsf{N}, g)$ is $(\mathsf{G},e,\mathfrak{m})$-naturally reductive. By \cite[Theorem 4.8]{MR0787113}, there exists a compact Lie algebra $(\mathfrak{k},\langle\cdot,\cdot\rangle_{\mathfrak{k}})$ endowed with a bi-invariant Euclidean scalar product and a finite-dimensional, orthogonal representation $\varphi: \mathfrak{k} \to \mathfrak{so}(V,\langle\cdot,\cdot\rangle_{V})$ without trivial submodules such that: \begin{itemize}
\item[$\bcdot$] $\mathfrak{n}$ is the real vector space given by the direct sum $\mathfrak{n} = \mathfrak{k} + V$, endowed with the Lie bracket $[\cdot,\cdot]_{\mathfrak{n}}$ whose only non-zero components are given by
$$
\langle [v_1,v_2]_{\mathfrak{n}}, K \rangle_{\mathfrak{k}} \coloneqq \langle \varphi(K)v_1,v_2 \rangle_{V}
\quad \text{for all } v_1,v_2 \in V,\ K \in \mathfrak{k};
$$
\item[$\bcdot$] $g$ is the scalar product defined by $g \coloneqq \langle\cdot,\cdot\rangle_{\mathfrak{k}} \oplus \langle\cdot,\cdot\rangle_{V}$;
\item[$\bcdot$] for every $K,K' \in \mathfrak{k}$ and $v \in V$, the map $\rho$ is given by:
\begin{equation} \label{eq:rhoGordon}
\rho(K)v = \varphi(K)v \,\, , \quad
\rho(K)K' = [K,K']_{\mathfrak{k}} \,\, , \quad
\rho(v) = 0 \,\, .
\end{equation}
\end{itemize}
Notice that $\mathfrak{z}(\mathfrak{n})$ coincides with the underlying vector space of $\mathfrak{k}$. By \eqref{eq:rhoGordon} and the fact that $\rho(\mathfrak{n}) \subset \mathfrak{u}(\mathfrak{n},J,g)$, it follows that $\varphi(\mathfrak{k}) \subset \mathfrak{u}(V,J|_V,\langle\cdot,\cdot\rangle_{V})$ and that $J|_{\mathfrak{k}}$ is $\mathrm{ad}(\mathfrak{k})$-invariant. Since $\mathfrak{k}$ is compact, the latter implies that it is abelian.

Conversely, assume now that $\mathfrak{z}(\mathfrak{n})$ is $J$-invariant and that $(\mathfrak{n},J|_{\mathfrak{z}(\mathfrak{n})^{\perp}},g)$ is a $\mathfrak{k}$-nilpotent triple, with $\mathfrak{k}$ abelian. We define then the linear map $\rho: \mathfrak{n} \to \mathfrak{gl}(\mathfrak{n})$ by \eqref{eq:rhoGordon} and we observe that, by hypothesis, it follows that $\rho(\mathfrak{n}) \subset \mathfrak{u}(\mathfrak{n},J,g) \cap \mathrm{Der}(\mathfrak{n})$, where $\mathrm{Der}(\mathfrak{n})$ denotes the space of derivations of $\mathfrak{n}$. By \cite[Theorem 4.8]{MR0787113}, the Riemannian manifold $(\mathsf{N},g)$ is $(\mathsf{G},e,\mathfrak{m})$-naturally reductive, where $\mathsf{G}$ is as in \eqref{eq:Gnilbas} and $\mathfrak{m}$ is given by \eqref{eq:nilbascompl}.  Since $\mathsf{G}$ acts by holomorphic isometries on $(\mathsf{N}, J, g)$, it follows that $(\mathsf{N}, J, g)$ is BAS by Theorem \ref{thm:BASnatrad}.
\end{proof}

In view of Theorem \ref{thm:nilBAS}, in order to classify simply-connected nilpotent Lie groups that admit a left-invariant BAS structure, it is sufficient to classify $\mathfrak{k}$-nilpotent Lie algebras with $\mathfrak{k}$ abelian.

\begin{proposition} \label{prop:BASnilalg}
Let $\mathfrak{n}$ be a Lie algebra. Then $\mathfrak{n}$ is $\mathfrak{k}$-nilpotent, with $\mathfrak{k}$ abelian, if and only if it admits a basis
$$
\mathfrak{n}=\mathrm{span}(z_1,\dots,z_{r}) + \mathrm{span}(e_1,\dots,e_{2m})
$$
such that the only non-zero Lie brackets are
\begin{equation} \label{eq:BASnilalg}
[e_{2j-1},e_{2j}] = \sum_{i=1}^{r}\lambda_j^i z_i \,\, , \quad j=1,\dots,m,
\end{equation}
for some coefficients $\lambda_j^i \in \mathbb{R}$, with $(\lambda_j^1,\dots,\lambda_j^{r}) \neq 0$ for every $1 \leq j \leq m$.
\end{proposition}

\begin{proof}
Let $\varphi: \mathbb{R}^r \to \mathfrak{u}(m)$ be a unitary representation of $\mathbb{R}^r$ on $\mathbb{C}^m$ without trivial submodules and denote by $\mathfrak{n}$ the corresponding $\mathbb{R}^r$-nilpotent Lie algebra as in Definition \ref{def:k-nil}. Fix an orthonormal basis $(z_1,{\dots},z_r)$ for $\mathbb{R}^r$ and notice that, since $\mathbb{R}^r$ is abelian, the endomorphisms $\varphi(z_i)$ can be simultaneously diagonalized. Therefore, there exists a unitary basis $(u_1,{\dots},u_m)$ for $\mathbb{C}^m$ with respect to which
$$
\varphi(z_i) = \left(\begin{array}{ccc}
\mathtt{i}\lambda^i_1 & & \\
 & \ddots & \\
 &  &  \mathtt{i}\lambda^i_m
\end{array}\right) \,\, , \quad \text{with $\lambda^i_j \in \mathbb{R}$} \,\, ,
$$
for every $1 \leq i \leq r$. Note that, since $\varphi$ has no trivial submodule, each $r$-tuple $(\lambda_j^1,\dots,\lambda_j^{r})$ is non-zero. Consider then the unitary real basis $\{e_1,{\dots}e_{2m}\}$ for $\mathbb{R}^{2m} = (\mathbb{C}^m)^{\mathbb{R}}$ given by
$$
e_{2i-1} \coloneqq \tfrac{1}{\sqrt{2}}(u_i +\bar{u}_i) \,\, , \quad e_{2i} \coloneqq \tfrac{\mathtt{i}}{\sqrt{2}}(u_i -\bar{u}_i) \,\, , \quad 1 \leq i \leq m \,\, .
$$
From \eqref{eq:knilbrack}, it follows that the only non-zero Lie brackets on $\mathfrak{n}$ are as in \eqref{eq:BASnilalg}.

Conversely, declare the given basis to be orthonormal and set $V \coloneqq \operatorname{span}\{e_1,{\dots},e_{2m}\}$, endowed with the complex structure given by $Ie_{2i-1} \coloneqq e_{2i}$. Define $\varphi : \mathbb{R}^r \to \mathfrak{u}(V,I,\langle\cdot,\cdot\rangle_V)$ by
$$
\varphi(z_i).e_{2j-1} \coloneqq \lambda_j^i e_{2j} \,\, , \quad \varphi(z_i).e_{2j} \coloneqq -\lambda_j^i e_{2j-1} \,\, .
$$
It is then straightforward to check that $\varphi$ is a finite-dimensional, unitary representation without trivial submodules, and that the corresponding $\mathbb{R}^r$-nilpotent Lie algebra is $\mathfrak{n}$.
\end{proof}

\begin{remark}
By means of Theorem \ref{thm:nilBAS} and Proposition \ref{prop:BASnilalg}, one can characterise compact nilmanifolds that admit an invariant BAS structure as follows. Let $\mathsf{N}$ be a simply-connected nilpotent Lie group, with Lie algebra $\mathfrak{n}$. Then, $\mathsf{N}$ admits a left-invariant BAS structure if and only if $\mathfrak{n}$ admits a basis with structure constants as in \eqref{eq:BASnilalg}, with $r$ even. If each coefficient $\lambda_j^i$ is rational, then the Lie group $\mathsf{N}$ admits a cocompact lattice $\Lambda \subset \mathsf{N}$ (see, \emph{e.g.}, \cite[Theorem 2.12]{MR507234}), and the quotient $\Lambda\backslash\mathsf{N}$ is a compact BAS nilmanifold.
\end{remark}

As a simple consequence of Theorem \ref{thm:nilBAS} and Proposition \ref{prop:BASnilalg}, it is possible to list nilpotent Lie algebras that admit a BAS structure in low dimension.

\begin{corollary} \label{cor:BASnil}
Let $\mathsf{N}$ be a non-abelian, nilpotent, simply-connected, $n$-dimensional Lie group, with $n \in \{4,6,8\}$. Then $\mathsf{N}$ admits a left-invariant BAS structure if and only if its Lie algebra $\mathfrak{n}$ is one of the Lie algebras in the following list.
\begin{itemize}
    \item[$i)$] If $n =4$, then $\mathfrak{n} = \mathbb{R} \oplus \mathfrak{h}_3(\mathbb{R})$.
    \item[$ii)$] If $n =6$ and 
    \begin{itemize}
        \item[$\bcdot$] $\mathrm{dim}(\mathfrak{z}(\mathfrak{n})) = 4$, then $\mathfrak{n} = \mathbb{R}^3 \oplus \mathfrak{h}_3(\mathbb{R})$;
        \item[$\bcdot$] $\mathrm{dim}(\mathfrak{z}(\mathfrak{n})) = 2$, then $\mathfrak{n} = \mathbb{R} \oplus \mathfrak{h}_5(\mathbb{R})$ or $\mathfrak{n} = \mathfrak{h}_3(\mathbb{R}) \oplus \mathfrak{h}_3(\mathbb{R})$.
    \end{itemize}
    \item[$iii)$] If $n =8$ and 
    \begin{itemize}
        \item[$\bcdot$] $\mathrm{dim}(\mathfrak{z}(\mathfrak{n})) = 6$, then    $\mathfrak{n} = \mathbb{R}^5 \oplus \mathfrak{h}_3(\mathbb{R})$;
        \item[$\bcdot$] $\mathrm{dim}(\mathfrak{z}(\mathfrak{n})) = 4$, then $\mathfrak{n} = \mathbb{R}^3 \oplus \mathfrak{h}_5(\mathbb{R})$, or $\mathfrak{n} = \mathbb{R}^2 \oplus \mathfrak{h}_3(\mathbb{R}) \oplus \mathfrak{h}_3(\mathbb{R})$;
        \item[$\bcdot$] $\mathrm{dim}(\mathfrak{z}(\mathfrak{n})) = 2$, then $\mathfrak{n} = \mathbb{R} \oplus \mathfrak{h}_7(\mathbb{R})$, $\mathfrak{n} = \mathfrak{h}_3(\mathbb{R}) \oplus \mathfrak{h}_5(\mathbb{R})$ or $\mathfrak{n} = \mathfrak{n}^8_{1,1}$ as in Example \ref{ex:n11}.
    \end{itemize}
\end{itemize}
\end{corollary}

\begin{proof}
Thanks to Theorem \ref{thm:nilBAS} and Proposition \ref{prop:BASnilalg} it suffices to classify, up to Lie algebra isomorphism, the non-abelian $\mathfrak{k}$-nilpotent Lie algebras, with $\mathfrak{k}$ abelian and even-dimensional. Denote by $\varphi: \mathbb{R}^{2k} \to \mathfrak{u}(m)$ the corresponding unitary representation without trivial submodules.
\begin{itemize}
\item[$i)$] Assume that $n=4$. Then $k=1$ and $m=1$. Therefore, up to equivalence, $\varphi(t,s) = \mathtt{i}s$, and thus $\mathfrak{n} = \mathbb{R} \oplus \mathfrak{h}_3(\mathbb{R})$.
\item[$ii)$] Assume that $n=6$. Then $k \in \{1,2\}$. If $k=2$, then $m=1$ and so the previous argument yields $\mathfrak{n} = \mathbb{R}^3 \oplus \mathfrak{h}_3(\mathbb{R})$. If $k = 1$, then $m = 2$ and so, up to equivalence, either
$$
\varphi(t,s) = \left(\begin{array}{cc}
\mathtt{i} s & 0 \\
0 & \mathtt{i}\alpha s
\end{array}\right) \quad \text{for $\alpha \in \mathbb{R} \setminus \{0\}$} \,\, ,
$$
yielding $\mathfrak{n} = \mathbb{R} \oplus \mathfrak{h}_5(\mathbb{R})$ up to Lie algebra isomorphism, or
$$
\varphi(t,s) = \left(\begin{array}{cc}
\mathtt{i}t & 0 \\
0 & \mathtt{i}s
\end{array}\right) \,\, ,
$$
yielding $\mathfrak{n} = \mathfrak{h}_3(\mathbb{R}) \oplus \mathfrak{h}_3(\mathbb{R})$.
\item[$iii)$] Assume that $n=8$. Then $k \in \{1,2,3\}$. If $k = 3$, then $m=1$ and the previous argument yields $\mathfrak{n} = \mathbb{R}^5 \oplus \mathfrak{h}_3(\mathbb{R})$. If $k = 2$, then $m = 2$ and so the previous argument yields $\mathfrak{n} = \mathbb{R}^3 \oplus \mathfrak{h}_5(\mathbb{R})$ or $\mathfrak{n} = \mathbb{R}^2 \oplus \mathfrak{h}_3(\mathbb{R}) \oplus \mathfrak{h}_3(\mathbb{R})$. If $k = 1$, then $m=3$ and so, up to equivalence, either
$$
\varphi(t,s) = \left(\begin{array}{ccc}
\mathtt{i} s & 0 & 0 \\
0 & \mathtt{i}\alpha s & 0 \\
0 & 0 & \mathtt{i}\beta s
\end{array}\right) \quad \text{for $\alpha, \beta \in \mathbb{R} \setminus \{0\}$} \,\, ,
$$
yielding $\mathfrak{n} = \mathbb{R} \oplus \mathfrak{h}_7(\mathbb{R})$ up to Lie algebra isomorphism, or
$$
\varphi(t,s) \coloneqq 
\left( \begin{array}{ccc}
\mathtt{i}t & 0 & 0 \\
0 & \mathtt{i}s & 0 \\
0 & 0 & \mathtt{i} (\alpha t+\beta s) \\
\end{array} \right) \quad \text{for $\alpha, \beta \in \mathbb{R}$, $\alpha^2 + \beta^2 \neq 0$} \,\, ,
$$
yielding $\mathfrak{n} = \mathfrak{h}_3(\mathbb{R}) \oplus \mathfrak{h}_5(\mathbb{R})$ or $\mathfrak{n} = \mathfrak{n}^8_{1,1}$ up to Lie algebra isomorphism.
\end{itemize}
\end{proof}

We are finally ready to complete the proof of Theorem \ref{thm:MAIN-nil}.

\begin{proof}[Proof of Theorem \ref{thm:MAIN-nil}]
The first claim of Theorem \ref{thm:MAIN-nil} follows directly from Theorem \ref{thm:nilBAS} and Proposition \ref{prop:BASnilalg}. Assume now that $(\mathsf{N},J,g)$ is BAS and observe that, by the proof of Theorem \ref{thm:nilBAS}, it is $(\mathsf{G},e,\mathfrak{m})$-naturally reductive, where $\mathsf{G}$ is as in \eqref{eq:Gnilbas} and $\mathfrak{m}$ is given by \eqref{eq:nilbascompl}. We also denote by $\mathfrak{g}$ the Lie algebra of $\mathsf{G}$.

We first observe that $\mathfrak{z}(\mathfrak{n}) = [\mathfrak{n},\mathfrak{n}] +J[\mathfrak{n},\mathfrak{n}]$. Indeed, the $g$-orthogonal complement $\mathfrak{d}$ of $[\mathfrak{n},\mathfrak{n}] +J[\mathfrak{n},\mathfrak{n}]$ in $\mathfrak{z}(\mathfrak{n})$ is a $J$-invariant, central ideal of $\mathfrak{n}$, hence it integrates to a flat K{\"a}hler--de Rham factor. Since $(\mathsf{N},J,g)$ does not split as a non-trivial Hermitian product, it then follows that $\mathfrak{d} = \{0\}$.

Next, we claim that $\mathfrak{h} = \mathfrak{u}(\mathfrak{n},J,g) \cap \mathrm{Der}(\mathfrak{n})$ acts trivially on $\mathfrak{z}(\mathfrak{n})$. Fix a unitary basis $\{e_i,Je_i\}_{1 \leq i \leq m}$ for $\mathfrak{z}(\mathfrak{n})^{\perp}$ as in Proposition \ref{prop:BASnilalg}. Let $D \in \mathfrak{h}$ and write
$$
De_i = \sum_{j}(a_{ij}e_j +b_{ij}Je_j) \quad \text{for some $a_{ij}, b_{ij} \in \mathbb{R}$} \,\, .
$$
By \eqref{eq:BASnilalg}, we obtain
$$
D[e_i,Je_i] = [De_i,Je_i] +[e_i,JDe_i] = 2a_{ii}[e_i,Je_i] \quad \text{for all $1 \leq i \leq m$} \,\, .
$$
Since $D$ is skew-symmetric, it follows that $a_{ii} = 0$, hence $D[e_i,Je_i] = 0$. This proves that $D|_{\mathfrak{z}(\mathfrak{n})} = 0$.

Since $\mathfrak{g} = \mathfrak{n} \rtimes \mathfrak{h}$, it follows that the centre of $\mathfrak{z}(\mathfrak{g})$ of $\mathfrak{g}$ coincides with $\mathfrak{z}(\mathfrak{n})$. Indeed, $\mathfrak{z}(\mathfrak{n}) \subset \mathfrak{z}(\mathfrak{g})$ because $\mathfrak{h}$ acts trivially on $\mathfrak{z}(\mathfrak{n})$. On the other hand, if $(v,D) \in \mathfrak{z}(\mathfrak{g})$, then
$$
(0,0) = [(v,D),(w,0)] = ([v,w]+Dw,0) \quad \text{for all $w \in \mathfrak{n}$} \,\, .
$$
This implies that $D = -\mathrm{ad}(v)$. Hence, $D$ is both nilpotent and skew-symmetric, which implies that $D=0$, and so $v \in \mathfrak{z}(\mathfrak{n})$.

We then split $\mathfrak{m}$ as $\mathfrak{m} = \mathfrak{m}_0 +\mathfrak{m}_1$, with
$$
\mathfrak{m}_0 \coloneqq \{(K,\rho(K)) : K \in \mathfrak{z}(\mathfrak{n})\} \,\, , \quad
\mathfrak{m}_1 \coloneqq \{(v,0) : v \in \mathfrak{z}(\mathfrak{n})^{\perp}\} \,\, .
$$
Hence, $\mathfrak{m}_0$ coincides with the projection of the centre of $\mathfrak{g}$ onto $\mathfrak{m}$. Moreover, $\mathfrak{m}_1$ is $\mathrm{ad}(\mathfrak{h})$-invariant and does not contain any $\mathrm{ad}(\mathfrak{h})$-trivial submodule. Finally, being $\mathfrak{n}$ $2$-step nilpotent, the quotient $\mathfrak{n}/\mathfrak{z}(\mathfrak{n})$ is abelian, and so this concludes the proof.
\end{proof}

\medskip
\section{Other examples of Bismut--Ambrose--Singer manifolds} \label{sect:final} \setcounter{equation} 0

\subsection{Canonical connections} \label{sect:BAS3types} \hfill \par

In this section, we use the results of Section \ref{sect:compact}, Section \ref{sect:noncompact}, and Section \ref{sect:nil} to provide canonical choices of Ambrose--Singer connections with totally skew-symmetric torsion on some classes of manifolds. These canonical connections are then used to construct new examples of BAS manifolds, generalizing all known examples in the literature.

\begin{remark} \label{rem:Gcompt}
Let $\pi: \mathsf{G}/\mathsf{H} \to \mathsf{G}/\mathsf{H}\mathsf{T}$ be a $\mathsf{G}$-equivariant principal torus bundle over a flag manifold, where $\mathsf{G}$ is a compact semisimple Lie group acting almost-effectively on $\mathsf{G}/\mathsf{H}$, and let $I$ be a $\mathsf{G}$-invariant complex structure on $\mathsf{G}/\mathsf{H}\mathsf{T}$. Assume that $g$ is a left-$\mathsf{G}$ and right-$\mathsf{T}$ invariant metric on $\mathsf{G}/\mathsf{H}$ inducing a normal metric on $\mathsf{G}/\mathsf{H}\mathsf{T}$. Then, there exists a unique $\mathrm{Ad}(\mathsf{G})$-invariant inner product $Q$ on the Lie algebra $\mathfrak{g}$ of $\mathsf{G}$ such that, for the corresponding $Q$-orthogonal Lie algebra decomposition
$$
\mathfrak{g} = \mathfrak{h} + \mathfrak{t} + \mathfrak{b} \,\, ,
$$
one has $Q|_{\mathfrak{b} \otimes \mathfrak{b}} = g|_{\mathfrak{b} \otimes \mathfrak{b}}$. Hence, by identifying $T_{e\mathsf{H}}\mathsf{G}/\mathsf{H} \simeq \mathfrak{t} + \mathfrak{b}$ and $T_{e\mathsf{H}\mathsf{T}}\mathsf{G}/\mathsf{H}\mathsf{T} \simeq \mathfrak{b}$ via the evaluation map, $I$ lifts to a $\mathsf{G}$-invariant tensor field on $\mathsf{G}/\mathsf{H}$ called \emph{transverse complex structure}. It is routine to check that the transverse complex structure satisfies
\begin{equation} \label{eq:trancs1}
[IX,IY]_{\mathfrak{t}} = [X,Y]_{\mathfrak{t}} \quad \text{for all $X,Y \in \mathfrak{b}$} \,\, .
\end{equation}
\end{remark}

\begin{remark} \label{rem:Gnoncompt}
Let $\pi: \mathsf{G}/\mathsf{H} \to \mathsf{G}/\mathsf{H}\mathsf{A}$ be a $\mathsf{G}$-equivariant principal abelian bundle over a Hermitian symmetric space, where $\mathsf{G}$ is a semisimple Lie group of non-compact type acting almost-effectively on $\mathsf{G}/\mathsf{H}$. Assume that $g$ is a left-$\mathsf{G}$ and right-$\mathsf{A}$ invariant metric on $\mathsf{G}/\mathsf{H}$ inducing the symmetric metric of $\mathsf{G}/\mathsf{H}\mathsf{A}$. We consider the reductive decomposition at Lie algebra level
$$
\mathfrak{g} = \mathfrak{h} +\mathfrak{a} +\mathfrak{b}
$$
orthogonal with respect to the Cartan-Killing form of $\mathfrak{g}$. Hence, by identifying $T_{e\mathsf{H}}\mathsf{G}/\mathsf{H} \simeq \mathfrak{a} + \mathfrak{b}$ and $T_{e\mathsf{H}\mathsf{A}}\mathsf{G}/\mathsf{H}\mathsf{A} \simeq \mathfrak{b}$ via the evaluation map, the complex structure $I$ of $\mathsf{G}/\mathsf{H}\mathsf{A}$ lifts to a $\mathsf{G}$-invariant tensor field on $\mathsf{G}/\mathsf{H}$ called \emph{transverse complex structure}. It is routine to check that the transverse complex structure satisfies
\begin{equation} \label{eq:trancs2}
[IX,IY]_{\mathfrak{a}} = [X,Y]_{\mathfrak{a}} \quad \text{for all $X,Y \in \mathfrak{b}$} \,\, .
\end{equation}
\end{remark}

The manifolds described above admit canonical choices of Ambrose--Singer connections with totally skew-symmetric torsion.

\begin{proposition} \label{prop:canGH}
Let $(\mathsf{G}/\mathsf{H},I,g)$ be a homogeneous space as in Remark \ref{rem:Gcompt} (resp.\ Remark \ref{rem:Gnoncompt}). Then, $(\mathsf{G}/\mathsf{H},g)$ admits an Ambrose--Singer connection $\nabla$ with totally skew-symmetric torsion $T$ that preserves the transverse complex structure $I$. Moreover, every fundamental vector field for the right-$\mathsf{T}$ (resp.\ right-$\mathsf{A}$) action is $\nabla$-parallel and the only non-zero components of $T$ are given by
\begin{equation} \label{eq:torsionGH}
T(X,Y,Z) = -g([X,Y]_{\mathfrak{b}}, Z) \,\, , \quad
T(X,Y,U) = -g([X,Y]_{\mathfrak{t}}, U)
\end{equation}
for all $X, Y, Z \in \mathfrak{b}$ and $U \in \mathfrak{t}$, with $\mathfrak{t}$ replaced by $\mathfrak{a}$ in the non-compact case.
\end{proposition}

\begin{proof}
Let $\mathfrak{v} \in \{\mathfrak{t},\mathfrak{a}\}$ denote the vertical abelian Lie subalgebra. Assume first that $\dim(\mathfrak{v})$ is even. Since $\mathfrak{v}$ is abelian and $[\mathfrak{v},\mathfrak{b}] \subset \mathfrak{b}$, it follows from \eqref{eq:trancs1} (resp.\ \eqref{eq:trancs2}) that every left-$\mathsf{G}$ and right-$\mathsf{T}$ (resp.\ right-$\mathsf{A}$) invariant almost complex structure on $\mathsf{G}/\mathsf{H}$ preserving $\mathfrak{v}$ and extending $I$ is integrable. We then fix one such extension $J$ which is also $g$-orthogonal. By Proposition \ref{prop:cmpctBAS} (resp.\ Proposition \ref{prop:noncmpctBAS}), $(\mathsf{G}/\mathsf{H},J,g)$ is BAS. A direct computation shows that the torsion $T = \mathrm{d}\omega(J\cdot,J\cdot,J\cdot)$ of its Bismut connection $\nabla$ is given by \eqref{eq:torsionGH}. Moreover, by using \eqref{eq:Bismutconn}, \eqref{eq:torsionGH} and \cite[Proposition 7.28]{MR2371700}, it follows that every fundamental vector field of the right-$\mathsf{T}$ (resp.\ right-$\mathsf{A}$) action is $\nabla$-parallel. By \eqref{eq:Bismutconn} and \eqref{eq:torsionGH}, the connection $\nabla$ is independent of the extension $J$ of $I$. This proves the claim for $\dim(\mathfrak{v})$ even.

If $\dim(\mathfrak{v})$ is odd, we apply the same argument as above to the product $\mathbb{R} \times \mathsf{G}/\mathsf{H}$, extending the metric $g$ in the obvious way. Let $U$ be a generator of the $\mathbb{R}$-factor and observe that, since $D U = \nabla U = 0$, we have $U \lrcorner T = 2(D U - \nabla U)=0$. Hence, the Bismut connection on $\mathbb{R} \times \mathsf{G}/\mathsf{H}$ induces an Ambrose--Singer connection $\nabla$ with totally skew-symmetric
torsion on $\mathsf{G}/\mathsf{H}$, whose torsion is given by \eqref{eq:torsionGH}. This concludes the proof. \end{proof}

With the same strategy, we can prove also the following analogous result for $\mathfrak{k}$-nilpotent triples.

\begin{proposition} \label{prop:canN}
Let $(\mathfrak{n},I,g)$ be a $\mathfrak{k}$-nilpotent triple as in Definition \ref{def:k-nil}, with $\mathfrak{k}$ abelian, and $\mathsf{N}$ the corresponding simply-connected Lie group. Then, $(\mathsf{N},g)$ admits an Ambrose--Singer connection $\nabla$ with totally skew-symmetric torsion $T$ that preserves the transverse complex structure $I$. Moreover, every central left-invariant vector field $U \in \mathfrak{z}(\mathfrak{n})$ is $\nabla$-parallel and the only non-zero components of $T$ are given by
\begin{equation} \label{eq:torsionN}
T(X,Y,U) = -g([X,Y],U)
\end{equation}
for all $X, Y \in \mathfrak{z}(\mathfrak{n})^{\perp}$ and $U \in \mathfrak{z}(\mathfrak{n})$.
\end{proposition}

\begin{proof}
Assume first that $\mathrm{dim}(\mathfrak{z}(\mathfrak{n}))$ is even. By \eqref{eq:abelianI}, every almost complex structure on $\mathfrak{n}$ that preserves the centre $\mathfrak{z}(\mathfrak{n})$ and extends $I$ is integrable. We then fix one such extensions $J$ which is also $g$-orthogonal. By Theorem \ref{thm:nilBAS}, $(\mathsf{N},J,g)$ is BAS. A direct computation shows that the torsion $T = \mathrm{d}\omega(J\cdot,J\cdot,J\cdot)$ of its Bismut connection $\nabla$ is given by \eqref{eq:torsionN}, and that $\nabla K = 0$ for every left-invariant vector field $K \in \mathfrak{z}(\mathfrak{n})$. By \eqref{eq:Bismutconn} and \eqref{eq:torsionN}, the connection $\nabla$ is independent of the extension $J$ of $I$. This proves the claim for $\mathrm{dim}(\mathfrak{z}(\mathfrak{n}))$ even.

If $\mathrm{dim}(\mathfrak{z}(\mathfrak{n}))$ is odd, we apply the same argument as above to the $\mathfrak{k}$-nilpotent Lie algebra $\mathbb{R} \oplus \mathfrak{n}$, extending the metric $g$ in the obvious way. Let us denote by $U$ a generator for the $\mathbb{R}$-factor and observe that, since $D U = \nabla U = 0$, we have $U \lrcorner T = 2(D U - \nabla U)=0$. Hence, the Bismut connection of $\mathbb{R} \times \mathsf{N}$ induces an Ambrose--Singer connection $\nabla$ with totally skew-symmetric torsion on $\mathsf{N}$, whose torsion is given by \eqref{eq:torsionN}. This concludes the proof.
\end{proof}

In the following, we construct examples of BAS manifolds combining a nilpotent factor with compact and non-compact homogeneous factors, corresponding to the three geometries investigated in the previous sections.

\begin{proposition} \label{prop:BAS-NGH}
Let $(\mathsf{N},I_0,g_0)$ be as in Proposition \ref{prop:canN}, let $(\mathsf{G}_1/\mathsf{H}_1,I_1,g_1)$ be as in Remark \ref{rem:Gcompt} and let $(\mathsf{G}_2/\mathsf{H}_2,I_2,g_2)$ be as in Remark \ref{rem:Gnoncompt}. Assume that the sum $\mathrm{dim}(\mathfrak{z}(\mathfrak{n})) + \mathrm{dim}(\mathfrak{t}) +\mathrm{dim}(\mathfrak{a})$ is even and consider the Riemannian product
$$
M \coloneqq \mathsf{N} \times \mathsf{G}_1/\mathsf{H}_1 \times \mathsf{G}_2/\mathsf{H}_2 \,\, , \quad g \coloneqq g_0 \oplus g_1 \oplus g_2 \,\, .
$$
Let $J$ be a $g$-orthogonal, left-$(\mathsf{N} \times \mathsf{G}_1 \times \mathsf{G}_2)$-invariant complex structure on $M$ that preserves the subspace $\mathfrak{z}(\mathfrak{n}) \oplus \mathfrak{t} \oplus \mathfrak{a}$ and restricts to $I_0$ on $\mathfrak{z}(\mathfrak{n})^{\perp}$, to $I_1$ on $\mathfrak{b}_1$ and to $I_2$ on $\mathfrak{b}_2$. Then, $(M,J,g)$ is BAS.
\end{proposition}

\begin{proof}
We consider the canonical connection $\nabla^0$ on $\mathsf{N}$ as in Proposition \ref{prop:canN}, the canonical connection $\nabla^i$ on $\mathsf{G}_i/\mathsf{H}_i$ as in Proposition \ref{prop:canGH}, for $i=1,2$, and we define $\nabla \coloneqq \nabla^0 \oplus \nabla^1 \oplus \nabla^2$ on $M$. By construction, $\nabla$ is an Ambrose--Singer connection on $(M,g)$ with totally skew-symmetric torsion such that $\nabla I_0 = \nabla I_1 = \nabla I_2 = 0$. By Proposition \ref{prop:canGH} and Proposition \ref{prop:canN}, the distribution of $TM$ corresponding to $\mathfrak{z}(\mathfrak{n}) \oplus \mathfrak{t} \oplus \mathfrak{a}$ is globally spanned by $\nabla$-parallel vector fields. Therefore, it follows that $\nabla J = 0$, and so $\nabla$ coincides with the Bismut connection of $(M,J,g)$, which concludes the proof.
\end{proof}

\begin{example}
By Theorem \ref{thm:nilBAS} and Proposition \ref{prop:BASnilalg}, all Kodaira manifolds are BAS. Moreover, in \cite{MR2062609} the authors construct left-invariant complex structures on the product $\mathsf{H}_{2n+1}(\mathbb{R}) \times \mathsf{G}$, where $\mathsf{H}_{2n+1}(\mathbb{R})$ is the real Heisenberg group and $\mathsf{G}$ is a simple odd-dimensional compact Lie group. These complex structures, combined with the canonical product metric, fit into the framework of Proposition \ref{prop:BAS-NGH}, and hence are BAS.
\end{example}

\begin{example}
All Calabi--Eckmann manifolds are BAS by Proposition \ref{prop:BAS-NGH}.
\end{example}

\subsection{Pluriclosed Bismut--Ambrose--Singer manifolds} \hfill \par

We recall that a Hermitian manifold $(M,J,g)$ is called {\it pluriclosed} if the torsion $T$ of its Bismut connection is closed. By exploiting the results in \cite{MR5124088}, we can classify Hermitian manifolds that are both BAS and pluriclosed. To this aim, we recall the following definition (see, \emph{e.g.}, \cite[Proposition 1.2]{MR2893680}).

\begin{definition}
A \emph{Sasaki manifold} is a triple $(S,g,\xi)$ consisting of an odd-dimensional Riemannian manifold $(S,g)$, together with a distinguished Killing vector field $\xi$ with unit norm, called {\it Reeb vector field}, such that
\begin{equation} \label{eq:Sasakidef}
\tfrac{1}{c^2}D^2_{X,Y}\xi = g(\xi,Y)X -g(X,Y)\xi \quad \text{for every $X,Y$} \,\, ,
\end{equation}
for some $c >0$. The distribution $\xi^{\perp} \subset TS$ is called the {\it transverse Sasaki distribution}, while the tensor field $\frac{1}{c}D\xi$ is called the {\it canonical transverse complex structure}.
\end{definition}

By \cite[Theorem 8.2]{MR1928632} (see also \cite[Theorem 2.2]{MR2322400}), every Sasaki manifold $(S,g,\xi)$ admits a metric connection with parallel totally skew-symmetric torsion that leaves both the Reeb vector field and the canonical transverse complex structure invariant. \smallskip

As a direct consequence of \cite[Theorem A]{MR5124088}, we get the following.

\begin{proposition} \label{prop:BAS-SKT}
Let $(M, J, g)$ be a complete, simply-connected, BAS manifold. Then $(M, J, g)$ is pluriclosed if and only if it decomposes as a product of Hermitian irreducible factors, each of which is either an irreducible Hermitian symmetric space or a Riemannian product
\begin{equation} \label{eq:BKL}
\mathbb{R}^{\ell} \times \prod_{i=1}^{s} S_{i} \times \mathsf{K}
\end{equation}
endowed with a standard complex structure, where $\mathbb{R}^{\ell}$ is the $\ell$-dimensional flat Euclidean space, each $S_{i}$ is a homogeneous Sasaki $3$-dimensional manifold, $\mathsf{K}$ is a compact semisimple Lie group of rank $r$ with a bi-invariant metric and $\ell \leq s+r$.
\end{proposition}

Before proving this result we shall recall that by \emph{standard complex structure} on \eqref{eq:BKL}, we mean an invariant complex structure $J$ that is compatible with the transverse complex distribution on each Sasaki factor, projects onto a $\mathsf{K}$-invariant complex structure on the full flag manifold $\mathsf{K}/\mathsf{T}$ for a chosen maximal torus $\mathsf{T} \subset \mathsf{K}$ and preserves the distribution spanned by the Euclidean factor, the Reeb vector fields and the torus $\mathsf T$ (see \cite[Definition 3.2]{MR5124088} for more details).

\begin{proof}[Proof of Proposition \ref{prop:BAS-SKT}]
Assume that $(M, J, g)$ is a complete, simply-connected, irreducible Hermitian non-K{\"a}hler manifold. If $(M, J, g)$ is BAS and pluriclosed, then the claim follows from \cite[Theorem A]{MR5124088}. Conversely, assume that $(M, J, g)$ is of the form \eqref{eq:BKL}. By \cite[Proposition 3.4]{MR5124088}, the Bismut connection $\nabla$ of $(M, J, g)$ has parallel torsion, decomposes according to the de Rham factors of $(M,g)$, and induces on each factor its canonical connection. On $\mathbb{R}^{\ell} \times \mathsf{K}$, the connection $\nabla$ is flat. Moreover, since each factor $S_i$ is a $3$-dimensional Sasaki manifold, the curvature $R_i$ of the induced connection $\nabla^i$ on $S_i$ is of the form
$$
R_i = s_i \,\mathrm{d}\xi_i^* \otimes \mathrm{d}\xi_i^*
$$
for some scalar function $s_i$, which is constant by homogeneity (see, \emph{e.g.}, \cite[Lemma 4.6]{MR5124088}). Since $\nabla^i\mathrm{d}\xi_i^* = 0$, it follows that $\nabla^i$ has parallel curvature, and this concludes the proof.
\end{proof}

We remark that every simply-connected, homogeneous, Sasaki $3$-dimensional manifold is equivariantly diffeomorphic to one of the following: $\mathsf{H}_3(\mathbb{R})$, $\mathsf{SU}(2)$, or the universal cover of $\mathsf{SL}(2,\mathbb{R})$ (see \cite[Theorem~1]{MR1464186}). Therefore, in the setting of Proposition \ref{prop:BAS-NGH}, these three cases correspond respectively to the nilpotent case, the compact case, and the semisimple non-compact case, and the standard complex structure agrees with the one defined there.

\end{document}